\documentclass[]{amsart}

\pdfoutput=1

\usepackage[utf8]{inputenc}
\usepackage{amsmath,amsthm,amssymb,blkarray}
\usepackage{amsfonts}
\usepackage{mathrsfs}
\usepackage{graphicx}
\usepackage{url}
\usepackage{array}
\usepackage{float}
\usepackage{multicol}
\usepackage{tabu}
\usepackage[table]{xcolor}
\usepackage{tikz}
\usetikzlibrary{decorations.markings}
\usetikzlibrary{decorations.pathreplacing}
\usetikzlibrary{arrows,shapes,positioning,backgrounds}
\usetikzlibrary{calc,trees}
\usepackage{nameref} 
\usepackage{caption}
\usepackage[labelformat=simple,labelfont={}]{subcaption}
	
\usetikzlibrary{shapes.geometric}
\usepackage[margin=1in]{geometry}
\usepackage{mathtools}
\usepackage[shortlabels]{enumitem}
\usepackage{color}
\usepackage{manfnt} 
\usepackage{MnSymbol}
\usepackage[breaklinks]{hyperref}
\usepackage{multicol}
\usepackage{harpoon}
\usepackage{subcaption}
\usepackage{stmaryrd}
\usepackage{calc}
\usepackage{underoverlap}
\usepackage{accents}
\usepackage{amsbsy}

\newcommand{\N}{\mathbb{N}}
\newcommand{\un}{\underline}
\DeclareMathOperator{\supp}{supp}
\DeclareMathOperator{\card}{card}
\DeclareMathOperator{\rank}{rank}
\DeclareMathOperator{\im}{im}

\renewcommand{\a}{\boldsymbol{a}}

\newcommand{\llb}{\llbracket}
\newcommand{\rrb}{\rrbracket}
\newcommand{\ralpha}{{\boldsymbol{\alpha}}}
\newcommand{\rbeta}{{\boldsymbol{\beta}}}
\newcommand{\rgamma}{{\boldsymbol{\gamma}}}
\newcommand{\rdelta}{{\boldsymbol{\delta}}}
\newcommand{\rsigma}{{\boldsymbol{{\sigma}}}}
\newcommand{\rphi}{{\boldsymbol{\varphi}}}
\newcommand{\rpsi}{{\boldsymbol{\psi}}}
\newcommand{\hatrphi}{{\hat{\boldsymbol{\varphi}}}}

\newcommand{\hatralpha}{{\hat{\boldsymbol{\alpha}}}}
\newcommand{\hatrbeta}{{\hat{\boldsymbol{\beta}}}}
\newcommand{\hatrgamma}{{\hat{\boldsymbol{\gamma}}}}
\newcommand{\hatrsigma}{{\hat{\boldsymbol{\sigma}}}}

\makeatletter
\newcommand{\hathat}[1]{%
\begingroup%
  \let\macc@kerna\z@%
  \let\macc@kernb\z@%
  \let\macc@nucleus\@empty%
  \hat{\mathchoice%
    {\raisebox{.4ex}{\vphantom{\ensuremath{\displaystyle #1}}}}%
    {\raisebox{.4ex}{\vphantom{\ensuremath{\textstyle #1}}}}%
    {\raisebox{.3ex}{\vphantom{\ensuremath{\scriptstyle #1}}}}%
    {\raisebox{.14ex}{\vphantom{\ensuremath{\scriptscriptstyle #1}}}}%
    \smash{\hat{#1}}}%
\endgroup%
}
\makeatother

\newcommand{\doublehatrphi}{{\hathat{\pmb{\varphi}}}}

\newcommand{\checkrbeta}{{\boldsymbol{\check{\beta}}}}
\newcommand{\hatcheckrbeta}{{\boldsymbol{\hat{\check{\beta}}}}}
\newcommand{\checkrgamma}{{\boldsymbol{\check{\gamma}}}}

\DeclareMathOperator{\bs}{\mathcal{S}}
\DeclareMathOperator{\R}{\mathcal{R}}

\DeclareMathOperator{\G}{\mathcal{G}}

\DeclareMathOperator{\F}{\mathcal{F}}

\newcommand{\lfrac}[2]{\genfrac{}{}{1pt}{}{#1}{\makebox[20pt][c]{$\scriptstyle
#2$}}}

\definecolor{orange}{RGB}{255,102,0}
\definecolor{ggreen}{RGB}{0,153,0}
\definecolor{darkblue}{RGB}{0,0,255}
\definecolor{purple}{RGB}{153,51,255}
\definecolor{turq}{RGB}{72,209,204}
\definecolor{gray}{RGB}{220,220,220}
\definecolor{orange2}{RGB}{255,100,0}
\definecolor{purple2}{RGB}{159,51,250}
\definecolor{rred}{rgb}{0.9, 0.17, 0.31}
\definecolor{naugreen}{cmyk}{.43,0,.34,.38}
\definecolor{naublue}{cmyk}{1,.72,0,.32}
\definecolor{mediterranean}{cmyk}{.67,0,.08,.3}
\definecolor{rose}{cmyk}{0,1.00,.20,0}
\definecolor{darkorchid}{cmyk}{.6,.9,0,.05}
\definecolor{butterfly}{cmyk}{.95,.59,0,.10}
\definecolor{springgreen}{cmyk}{1.00,0,.70,.02}
\definecolor{darkred}{cmyk}{0,1,1,.5}
\definecolor{nectarine}{cmyk}{0,0.70,1.00,0}
\definecolor{icyblue}{cmyk}{.84,.25,0,.06}
\definecolor{manatee}{rgb}{0.59, 0.6, 0.67}

\newcommand{\textover}[3][l]{%
 \makebox[\widthof{#3}][#1]{#2}%
 }
\makeatletter
\newcommand{\subword}[1]{%
\cdots \ \underset{\textover[c]{{$\scriptstyle i$}}{$\scriptstyle
+++$}}{\un{\textover[c]{$#1$}{$s$}}}\checknextarg}
\newcommand{\checknextarg}{\@ifnextchar\bgroup{\gobblenextarg}{\ \cdots \ }}
\newcommand{\gobblenextarg}[2]{\underset{\textover[c]{$\scriptstyle
#1$}{$\scriptstyle
+++$}}{\un{\textover[c]{$#2$}{$s$}}}\@ifnextchar\bgroup{\gobblenextarg}{ \
\cdots }}
\makeatother

\theoremstyle{definition}
\newtheorem{theorem}{Theorem}[section]
\newtheorem{corollary}[theorem]{Corollary}
\newtheorem{lemma}[theorem]{Lemma}
\newtheorem{conjecture}[theorem]{Conjecture}
\newtheorem{definition}[theorem]{Definition}
\newtheorem{proposition}[theorem]{Proposition}
\newtheorem{example}[theorem]{Example}
\newtheorem{remark}[theorem]{Remark}

\pgfdeclarelayer{background}
\pgfsetlayers{background,main}

\tikzset{
    my box/.style = {
        , line cap = round
        , line join = round
    }
}

\newcommand{\highlight}[3]{
\path [my box, line width = #1, draw = #2,opacity=.2] #3;
}

\begin{document}

\title{Braid graphs in simply-laced triangle-free Coxeter systems are partial
cubes}
\author{Fadi Awik, Jadyn Breland, Quentin Cadman, and Dana C.~Ernst}
 
\address{
Department of Mathematics and Statistics,
Northern Arizona University PO Box 5717,
Flagstaff, AZ 86011
}
\email{faa234@nau.edu, qdc4@nau.edu, Dana.Ernst@nau.edu}

\address{
Mathematics Department,
University of California Santa Cruz,
1156 High Street,
Santa Cruz, CA 95064
}
\email{jbreland@ucsc.edu}


\begin{abstract}
In this paper, we study the structure of braid graphs in simply-laced Coxeter
systems. We prove that every reduced expression has a unique factorization as a
product of so-called links, which in turn induces a decomposition of the braid
graph into a box product of the braid graphs for each link factor. When the
Coxeter graph has no three-cycles, we use the decomposition to prove that braid
graphs are partial cubes, i.e., can be isometrically embedded into a hypercube.
For a special class of links, called Fibonacci links, we prove that the
corresponding braid graphs are Fibonacci cubes.
\end{abstract}

\maketitle

\section{Introduction}

Every element $w$ of a Coxeter group $W$ can be written as an expression in the
generators, and if the number of generators in an expression (including
multiplicity) is minimal, we say that the expression is reduced. According to
Matsumoto's Theorem~\cite[Theorem~1.2.2]{Geck2000}, every pair of reduced
expressions for the same group element are related by a sequence of so-called
commutation and braid moves. In light of Matsumoto's Theorem, we can define a
connected graph on the set of reduced expressions of a given element in a
Coxeter group. We define the Matsumoto graph of $w\in W$ to be the graph having
vertex set equal to the set of reduced expressions of $w$, where two vertices
are connected by an edge if and only if the corresponding reduced expressions
are related by a single commutation or braid move. Bergeron, Ceballos, and
Labb\'e~\cite{Bergeron2015} proved that for finite Coxeter groups, every cycle
in a Matsumoto graph has even length. In~\cite{Grinberg2017}, Grinberg and
Postnikov extended this result to arbitrary Coxeter systems. Since every cycle
in a Matsumoto graph has even length, every Matsumoto graph is always
bipartite.

Two reduced expressions for the same Coxeter group element are said to be
commutation equivalent if we can obtain one from the other via a sequence of
commutation moves. The corresponding equivalence classes are referred to as
commutation classes, and have been studied extensively in the literature. In
the case of Coxeter systems of type $A_n$, Elnitsky~\cite{Elnitsky1997} showed
that the set of commutation classes for a given permutation $w$ is in
one-to-one correspondence with the set of rhombic tilings of a certain polygon
determined by $w$. Meng~\cite{Meng2010} studied the number of commutation
classes and their relationships via braid moves, and B\'edard~\cite{Bedard1999}
developed recursive formulas for the number of reduced expressions in each
commutation class. According to~\cite{Humphreys1990}, every finite Coxeter
group contains a unique element of maximal length, called the longest element.
Determining the number of commutation classes for the longest element in
Coxeter systems of type $A_n$ remains an open problem. To our knowledge, this
problem was first introduced in 1992 by Knuth in Section 9 of~\cite{Knuth1992}
using different terminology. In the paragraph following the proof of Proposition~4.4 of~\cite{Tenner2006}, Tenner explicitly states the open problem in terms of commutation classes. Even less is known about the number of commutation classes of the longest elements in other finite Coxeter groups.

Similarly, we define two reduced expressions to be braid equivalent if they are
related by a sequence of braid moves, where the corresponding equivalence
classes are called braid classes. Braid classes have appeared in the work of
Bergeron, Ceballos, and Labb\'e~\cite{Bergeron2015} while
Zollinger~\cite{Zollinger1994a} provided formulas for the cardinality of braid
classes in the case of Coxeter systems of type $A_n$. Fishel et
al.~\cite{Fishel2018} provided upper and lower bounds on the number of reduced
expressions for a fixed permutation in Coxeter systems of type $A_n$ by
studying the commutation classes and braid classes in tandem. However, unlike
commutation classes, braid classes have received very little attention.  Many
natural questions regarding braid classes remain open, even for type $A_n$. For
example, how many braid classes occur for elements of a fixed length? In
particular, which elements for a fixed length have maximally many braid
classes? And how many braid classes does the longest element in the Coxeter
system of type $A_n$ have? Answering the latter question might provide insight
into the analogous question for commutation classes that was mentioned above.

The relationship among the reduced expressions in a fixed braid class can be
encoded in a graph. Define the braid graph for a reduced expression to be the
graph with vertex set equal to the corresponding braid class, where two
vertices are connected by an edge if and only if the corresponding reduced
expressions are related via a single braid move.  Note that every braid graph
is equal to one of the connected components of the graph obtained by deleting
the edges corresponding to commutation moves in the Matsumoto graph for the
corresponding group element. The overarching goal of this paper is to
understand the structure of braid classes in simply-laced Coxeter systems by
studying the combinatorial architecture of the corresponding braid graphs.

We begin by recalling the basic terminology of Coxeter systems and establish
our notation in Section~\ref{sec:prelims}. In Section~\ref{sec:architecture},
for simply-laced Coxeter systems, i.e., all braid moves are of the form
$sts\mapsto tst$, we introduce the notion of a braid shadow, which is the
location where a braid move may be applied.  Loosely speaking, we say that a
reduced expression is a link if it consists of a single generator or there is a
sequence of overlapping braid shadows that spans the positions of the reduced
expression. It turns out that every reduced expression can be written uniquely
as a product of maximal links.  We refer to this product as the link
factorization for the reduced expression.  One consequence of this
factorization is that every braid graph can be decomposed as the box product of
the braid graphs for the corresponding factors
(Corollary~\ref{cor:boxproductsofbraidgraphs}). In Section~\ref{sec:type A}, we
completely characterize links and braid graphs in Coxeter systems of type
$A_n$. Our notion of link coincides with the definition of string introduced by
Zollinger~\cite{Zollinger1994a} in the type $A_n$ case.
Section~\ref{sec:embedding} begins with a few technical lemmas and then
concludes with one of our main results (Theorem~\ref{thm:braid graphs are
partial cubes}), which states that every braid graph in a simply-laced Coxeter
system having no three-cycles in its Coxeter graph can be isometrically
embedded into a hypercube. Such graphs are called partial cubes. In
Section~\ref{sec:fibonacci}, we introduce a special class of links, called
Fibonacci links, and then show that the isometric embedding of the
corresponding braid graph into the hypercube is an isometry to a Fibonacci cube
graph (Theorem~\ref{thm:FibonacciCube}). We conclude with some open questions
in Section~\ref{sec:closing}.

\section{Preliminaries}\label{sec:prelims}

A \emph{Coxeter matrix} is an $n\times n$ symmetric matrix $M=(m_{ij})$ with
entries $m_{ij}\in\{1,2,3,\ldots,\infty\}$ such that $m_{ii}=1$ for all $1\leq
i\leq n$ and $m_{ij}\geq 2$ for $i\neq j$. A \emph{Coxeter system} is a pair
$(W,S)$ consisting of a finite set $S=\{s_1,s_2,\ldots,s_n\}$ and a group $W$,
called a \emph{Coxeter group}, with presentation
\[
W = \langle s_1,s_2,\ldots,s_n \mid (s_is_j)^{m(s_i,s_j)}=e\rangle,
\] 
where $m(s_i,s_j):=m_{ij}$ for some $n\times n$ Coxeter matrix $M=(m_{ij})$.
For $s,t\in S$, the condition $m(s,t)=\infty$ means that there is no relation
imposed between $s$ and $t$. It turns out that the elements of $S$ are distinct
as group elements and $m(s,t)$ is the order of $st$~\cite{Humphreys1990}. Since
elements of $S$ have order two, the relation $(st)^{m(s,t)} = e$ can be written
as
\[
\underbrace{sts \cdots}_{m(s,t)} = \underbrace{tst \cdots}_{m(s,t)}
\]
with $m(s,t)\geq 2$ letters. 
When $m(s,t)=2$, $st=ts$ is called a \emph{commutation relation} and when
$m(s,t) \geq 3$, the corresponding relation is called a \emph{braid relation}.
The replacement
\[
\underbrace{sts\cdots}_{m(s,t)} \longmapsto  \underbrace{tst\cdots}_{m(s,t)}
\]
is called a \emph{commutation move} if $m(s,t)=2$ and a \emph{braid move} if
$m(s,t) \geq 3$. We say that a Coxeter system is \emph{simply laced} provided
that $m(s,t)\leq 3$ for all $s,t\in S$. Our focus in this paper will be on
simply-laced Coxeter systems.

A Coxeter system $(W,S)$ can be encoded by a unique \emph{Coxeter graph}
$\Gamma$ having vertex set $S$ and edges $\{s,t\}$ for each $m(s,t)\geq 3$.
Moreover, each edge is labeled with the corresponding $m(s,t)$, although
typically the labels of $3$ are omitted because they are the most common. In
this case, we say that $(W,S)$, or just $W$, is of type $\Gamma$, and we may
denote the Coxeter group as $W(\Gamma)$ and the generating set as $S(\Gamma)$
for emphasis. In the case that $\Gamma$ has no three-cycles, we say that
$(W,S)$ is \emph{triangle free}.
 
\begin{example}
The Coxeter graphs given in Figure~\ref{fig:labeledgraphs} correspond to four
common simply-laced Coxeter systems. We summarize the defining relations for
the Coxeter systems determined by the graphs in
Figures~\ref{fig:labeledA}~and~\ref{fig:labeledB} below.
\begin{enumerate}[label=(\alph*)]
\item The Coxeter system of type $A_n$ is given by the Coxeter graph in
Figure~\ref{fig:labeledA}. In this case, $W(A_n)$ is generated by $S(A_n) =
\{s_1, s_2, \ldots, s_n\}$ and has defining relations
\begin{enumerate}[label=\arabic*.]
\item $s_i^2 = e$ for all $i$;
\item $s_is_j = s_js_i$ when $|{i-j}| > 1$;
\item $s_is_js_i = s_js_is_j$ when $|{i-j}| = 1$.
\end{enumerate}
The Coxeter group $W(A_n)$ is isomorphic to the symmetric group $S_{n+1}$ under
the mapping that sends $s_i$ to the adjacent transposition $(i,i+1)$.

\item\label{ex:B} The Coxeter system of type $D_n$ is given by the graph in
Figure~\ref{fig:labeledB}. The Coxeter group $W(D_n)$ has generating set
$S(D_n)=\{s_1,s_2, \ldots ,s_{n}\}$ and defining relations
\begin{enumerate}[label=\rm{\arabic*.}]
	\item $s_i^2=e$ for all $i$;
	\item $s_is_j = s_js_i$ if $|i-j| > 1$ and $i,j \neq 1$;
	\item $s_is_j=s_js_i$ if $i=1$ and $j \neq 3$;
	\item $s_1s_3s_1 = s_3s_1s_3$ and $s_is_js_i = s_js_is_j$ if $|i-j| = 1$.
\end{enumerate}
The Coxeter group $W(D_n)$ is isomorphic to the index two subgroup of the group
of signed permutations on $n$ letters having an even number of sign changes.
\end{enumerate}
With the exception of the type $\widetilde{A}_2$ Coxeter system, each of the
Coxeter systems described in Figure~\ref{fig:labeledgraphs} are triangle free.
\end{example}

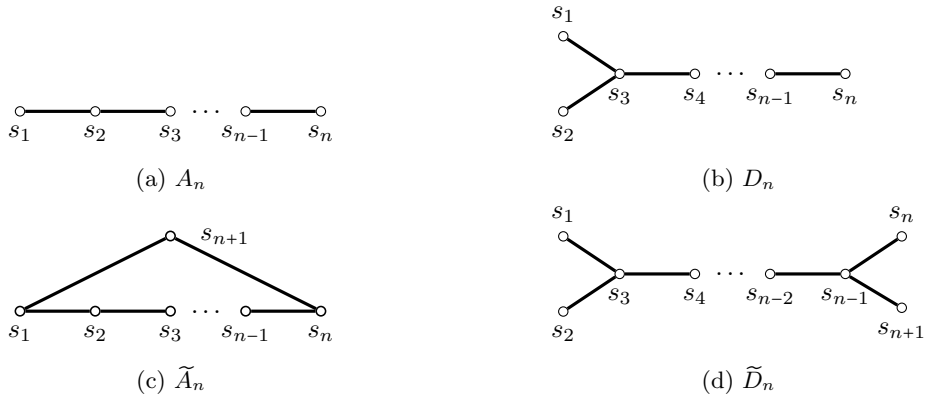
\begin{figure}[h]
\centering
\subcaptionbox{\label{fig:labeledA}$A_{n}$}[.45\linewidth]{
\begin{tikzpicture}[every circle node/.style={draw, circle, inner sep=1.25pt}]
\node [circle] (1) [label=below:$s_1$] at (0,0){};
\node [circle] (2) [label=below:$s_2$] at (1,0){};
\node [circle] (3) [label=below:$s_3$] at (2,0){};
\node at (2.5,0) {$\dots$};
\node [circle] (4) [label=below:$s_{n-1}$] at (3,0){};
\node [circle] (5) [label=below:$s_n$] at (4,0){};
\draw [-, very thick] (1) to (2);
\draw [-, very thick] (2) to (3);
\draw [-, very thick] (4) to (5);
\end{tikzpicture}
}
\subcaptionbox{\label{fig:labeledB}$D_n$}[.45\linewidth]{
\begin{tikzpicture}[every circle node/.style={draw, circle, inner sep=1.25pt}]
\node [circle] (1) [label=below:$s_2$] at (0.25,-0.5){};
\node [circle] (2) [label=below:$s_3$] at (1,0){};
\node [circle] (3) [label=below:$s_4$] at (2,0){};
\node at (2.5,0) {$\dots$};
\node [circle] (4) [label=below:$s_{n-1}$] at (3,0){};
\node [circle] (5) [label=below:$s_{n}$] at (4,0){};
\node [circle] (6) [label=above:$s_1$] at (.25,.5){};
\node [label=below:\phantom{$s_{n+1}$}](8) at (4.75,0){};
alignment
\draw [-, very thick] (1) to (2);
\draw [-, very thick] (2) to (3);
\draw [-, very thick] (4) to (5);
\draw [-, very thick] (2) to (6);
\end{tikzpicture}
}
\subcaptionbox{\label{fig:labeledC}$\widetilde{A}_n$}[.45\linewidth]{
\begin{tikzpicture}[every circle node/.style={draw, circle, semithick, inner
sep=1.25pt}]
\node [circle] (1) [label=below:$s_1$] at (0,0){};
\node [circle] (2) [label=below:$s_2$] at (1,0){};
\node [circle] (3) [label=below:$s_3$] at (2,0){};
\node at (2.5,0) {$\dots$};
\node [circle] (4) [label=below:$s_{n-1}$] at (3,0){};
\node [circle] (5) [label=below:$s_n$] at (4,0){};
\node [circle] (6) [label={[label distance=2mm]right:$s_{n+1}$}] at (2,1){};
\draw [-, very thick] (1) to (2);
\draw [-, very thick] (2) to (3);
\draw [-, very thick] (4) to (5);
\draw [-, very thick] (1) to (6);
\draw [-, very thick] (5) to (6);
\end{tikzpicture}
}
\subcaptionbox{\label{fig:labeledD}$\widetilde{D}_n$}[.45\linewidth]{
\begin{tikzpicture}[every circle node/.style={draw, circle, inner sep=1.25pt}]
\node [circle] (1) [label=below:$s_2$] at (0.25,-0.5){};
\node [circle] (2) [label=below:$s_3$] at (1,0){};
\node [circle] (3) [label=below:$s_4$] at (2,0){};
\node at (2.5,0) {$\dots$};
\node [circle] (4) [label=below:$s_{n-2}$] at (3,0){};
\node [circle] (5) [label=below:$s_{n-1}$] at (4,0){};
\node [circle] (6) [label=above:$s_1$] at (.25,.5){};
\node [circle] (7) [label=above:$s_{n}$] at (4.75,.5){};
\node [circle] (8) [label=below:$s_{n+1}$] at (4.75,-.45){};
\draw [-, very thick] (1) to (2);
\draw [-, very thick] (2) to (3);
\draw [-, very thick] (4) to (5);
\draw [-, very thick] (2) to (6);
\draw [-, very thick] (5) to (7);
\draw [-, very thick] (5) to (8);
\end{tikzpicture}
}
\caption{Examples of common simply-laced Coxeter
graphs.}\label{fig:labeledgraphs}
\end{figure}

Given a Coxeter system $(W,S)$, let $S^*$ denote the free monoid on the
alphabet $S$. An element $\ralpha = s_{x_1}s_{x_2}\cdots s_{x_m}\in S^* $ is
called a \emph{word} while a factor of $\ralpha$ is a word of the form
$s_{x_i}s_{x_{i+1}}\cdots s_{x_{j-1}}s_{x_j}$ for $1 \leq i \leq j \leq m$. If
a word $\ralpha=s_{x_1}s_{x_2}\cdots s_{x_m}\in S^*$ is equal to $w$ when
considered as a group element of $W$, we say that $\ralpha$ is an
\emph{expression} for $w$. If $m$ is minimal among all possible expressions for
$w$, we say that $\ralpha$ is a \emph{reduced expression} for $w$, and we call
$\ell(w):=m$ the \emph{length} of $w$. Note that any factor of a reduced
expression is also reduced. The set of all reduced expressions for $w\in W$
will be denoted by $\R(w)$. According to~\cite{Humphreys1990}, every finite
Coxeter group contains a unique element of maximal length, called the
\emph{longest element}, often denoted by $w_0$ if the context is clear.

For the remainder of this paper, if we are considering a particular labeling of
a Coxeter graph, we will often write $i$ in place of $s_i$ for brevity.

\begin{example}
It is well known that the longest element in $W(A_n)$ is given in one-line
notation by
\[
w_0 = [n+1, n, \ldots, 2, 1]
\]
and has length $\ell(w_0)=\binom{n+1}{2}$. One possible reduced expression for
$w_0\in W(A_n)$ is given by
\[
1\mid 21 \mid 321 \mid \cdots \mid n(n-1)\cdots 321,
\]
where the vertical bars have been placed to aid in recognizing the given
pattern. A formula for the number of reduced expressions for $w_0 \in W(A_n)$
is given in~\cite{Stanley1984}.
\end{example}

The following theorem, called Matsumoto's
Theorem~\cite[Theorem~1.2.2]{Geck2000}, characterizes  the relationship between
reduced expressions for a given group element.

\begin{proposition}[Matsumoto's Theorem]{\label{prop:matsumoto}}
In a Coxeter system $(W,S)$, any two reduced expressions for the same group
element differ by a sequence of commutation and braid moves.
\end{proposition}

In light of {Matsumoto's Theorem}, we can define a graph on the set of reduced
expressions of a given element in a Coxeter group.  For $w\in W$, define the
\emph{Matsumoto graph} $\G(w)$ to be the graph having vertex set equal to
$\R(w)$, where two reduced expressions $\ralpha$ and $\rbeta$ are connected by
an edge if and only if $\ralpha$ and $\rbeta$ are related via a single
commutation or braid move. Matsumoto's Theorem implies that $\G(w)$ is
connected. The graph obtained by contracting the edges corresponding to
commutation moves in the Matsumoto graph for the longest element in a Coxeter
systems of type $A_n$ has been studied by several authors, usually in the
context of the higher Bruhat order $B(n,2)$~\cite{Manin1989, Ziegler1993} or
rhombic tilings of polygons~\cite{Elnitsky1997}.

\begin{example}\label{ex:longest element An}
Consider the longest element $w_0$ in the Coxeter system of type $A_3$.  It
turns out that $\ell(w_0)=6$ and that there are 16 reduced expressions in
$\R(w_0)$. The corresponding Matsumoto graph is given in
Figure~\ref{fig:Matsumoto graph}. The 16 reduced expressions are the vertices
of $\G(w_0)$ and the edges show how pairs of reduced expressions are related
via commutation or braid moves. In order to distinguish between commutation and
braid moves, we have colored an edge \textcolor{orange2}{orange} if it
corresponds to a commutation move and \textcolor{turq}{blue} if it corresponds
to a braid move.
\end{example}

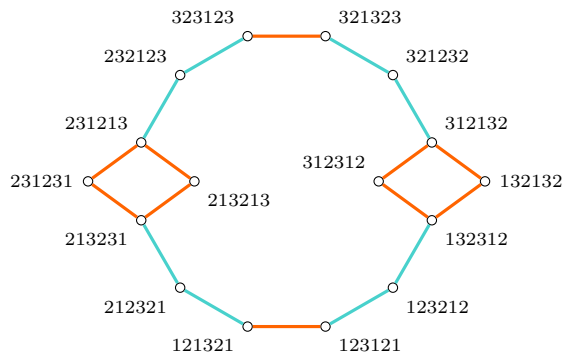
\begin{figure}[!ht]
\begin{center}
\begin{tikzpicture}[every circle node/.style={draw, circle, inner sep=1.25pt}]
\node [circle] (1) [label=below right: $\scriptstyle 132312$] at
({2*cos(15)},{-2*sin(15)}){};
\node [circle] (2) [label=above right: $\scriptstyle 312132$] at
({sqrt(3)*cos(15) + sin(15)},{-sqrt(3)*sin(15) + cos(15)}){};
\node [circle] (3) [label=above right: $\scriptstyle 321232$]at ({cos(15) +
sqrt(3)*sin(15)},{-sin(15) + sqrt(3)*cos(15)}){};
\node [circle] (4) [label=above right:$\scriptstyle 321323$] at
({2*sin(15)},{2*cos(15)}){};
\node [circle] (5) [label=above left:$\scriptstyle 323123$] at
({-1*cos(15)+sqrt(3)*sin(15)},{sin(15) +sqrt(3)*cos(15)}){};
\node [circle] (6) [label=above left:$\scriptstyle 232123$] at
({-sqrt(3)*cos(15) + sin(15)},{sqrt(3)*sin(15) + cos(15)}){};
\node [circle] (7) [label=above left:$\scriptstyle 231213$] at
({-2*cos(15)},{2*sin(15)}){};
\node [circle] (8) [label=below left:$\scriptstyle 213231$] at
({-sqrt(3)*cos(15) - sin(15)},{sqrt(3)*sin(15) -1*cos(15)}){};
\node [circle] (9) [label=below left:$\scriptstyle 212321$] at
({-cos(15)-sqrt(3)*sin(15)},{sin(15) -sqrt(3)*cos(15)}){};
\node [circle] (10) [label=below left:$\scriptstyle 121321$] at
({-2*sin(15)},{-2*cos(15)}){};
\node [circle] (11) [label=below right:$\scriptstyle 123121$] at ({cos(15) -
sqrt(3)*sin(15)},{-sin(15) -sqrt(3)*cos(15)}){};
\node [circle] (12) [label=below right:$\scriptstyle 123212$] at
({sqrt(3)*cos(15) - sin(15)},{-sqrt(3)*sin(15)-cos(15)}){};

\node [circle] (13) [label=right:$\scriptstyle 132132$] at
({2*cos(15)+.707},0){};
\node [circle] (14) [label=above left:$\scriptstyle 312312$] at
({2*cos(15)-.707},0){};
\node [circle] (15) [label=below right:$\scriptstyle 213213$] at
({-2*cos(15)+.707},0){};
\node [circle] (16) [label=left:$\scriptstyle 231231$]at
({-2*cos(15)-.707},0){};

\draw [turq,-, very thick] (2) to (3);
\draw [turq,-, very thick] (3) to (4);
\draw [orange2,-, very thick] (4) to (5);
\draw [turq,-, very thick] (5) to (6);
\draw [turq,-, very thick] (6) to (7);
\draw [turq,-, very thick] (8) to (9);
\draw [turq,-, very thick] (9) to (10);
\draw [orange2,-, very thick] (10) to (11);
\draw [turq,-, very thick] (11) to (12);
\draw [turq,-, very thick] (12) to (1);
\draw [orange2,-, very thick] (7) to (15);
\draw [orange2,-, very thick] (7) to (16);
\draw [orange2,-, very thick] (8) to (15);
\draw [orange2,-, very thick] (8) to (16);
\draw [orange2,-, very thick] (1) to (13);
\draw [orange2,-, very thick] (1) to (14);
\draw [orange2,-, very thick] (2) to (13);
\draw [orange2,-, very thick] (2) to (14);
\end{tikzpicture}
\setlength{\belowcaptionskip}{-6pt}
\caption{Matsumoto graph for the longest element in $W(A_3)$.}
\label{fig:Matsumoto graph}
\end{center}
\end{figure}

We now define two different equivalence relations on the set of reduced
expressions for a given element of a Coxeter group. Let $(W,S)$ be a Coxeter
system of type $\Gamma$ and let $w \in W$. For $\ralpha,\rbeta\in\R(w)$,
$\ralpha \sim_c \rbeta$ if we can obtain $\ralpha$ from $\rbeta$ by applying a
single commutation move. We define the equivalence relation $\approx_c$ by
taking the reflexive and transitive closure of $\sim_c$. Each equivalence class
under $\approx_c$ is called a \emph{commutation class}, denoted $[\ralpha]_c$
for $\ralpha\in\R(w)$. Two reduced expressions are said to be \emph{commutation
equivalent} if they are in the same commutation class.

Similarly, we define $\ralpha \sim_b \rbeta$ if we can obtain $\ralpha$ from
$\rbeta$ by applying a single braid move. The equivalence relation $\approx_b$
is defined by taking the reflexive and transitive closure of $\sim_b$. Each
equivalence class under $\approx_b$ is called a \emph{braid class}, denoted
$[\ralpha]_b$ for $\ralpha\in\R(w)$. Two reduced expressions are said to be
\emph{braid equivalent} if they are in the same braid class.

\begin{example}\label{ex:A3longest}
The set of 16 reduced expressions for the longest element in the Coxeter system
of type $A_3$ is partitioned into eight commutation classes and eight braid
classes:
\begin{equation*}
\begin{aligned}[c]
[232123]_c&=\{232123\}\\ 
[231213]_c&=\{231213,213213,213231,231231\}\\
[321323]_c&=\{321323,323123\}\\
[212321]_c&=\{212321\}\\
[321232]_c&=\{321232\}\\
[123123]_c&=\{123121,121321\}\\ 
[132312]_c&=\{132312,132132,312132,312312\}\\ 
[123212]_c&=\{123212\}
\end{aligned}
\qquad\qquad\qquad
\begin{aligned}[c]
[123121]_b&=\{123121,123212,132312\}\\ 
[312312]_b&=\{312312\}\\
[312132]_b&=\{312132,321232,321323\}\\
[132132]_b&=\{132132\}\\
[121321]_b&=\{121321,212321,213231\}\\
[213213]_b&=\{213213\}\\ 
[231213]_b&=\{231213,232123,323123\}\\ 
[231231]_b&=\{231231\}\end{aligned}
\end{equation*}
In general, it is not the case that the number of commutation classes for a
fixed group element is equal to the number of braid classes. Notice that the
four braid classes of size 3 correspond to the vertices in the
\textcolor{turq}{blue} connected components of the Matsumoto graph given in
Figure~\ref{fig:Matsumoto graph} while the singleton braid classes correspond
to the four vertices that are not incident to any \textcolor{turq}{blue} edges.
A similar structure holds for the commutation classes.
\end{example}

The remainder of this paper will focus exclusively on braid classes in
simply-laced Coxeter systems with an aim of describing their combinatorial
architecture. We will now write $[\ralpha]$ in place of $[\ralpha]_b$.

The relationships among the reduced expressions in a fixed braid class are
encoded by one of the maximal connected components of the underlying Matsumoto
graph consisting only of edges corresponding to braid moves. For example, each
braid class in Example~\ref{ex:A3longest} corresponds to one of the maximal
\textcolor{turq}{blue} connected components in the Matsumoto graph given in
Figure~\ref{fig:Matsumoto graph}. Let $\ralpha$ be a reduced expression for
$w\in W$. We define the \emph{braid graph} for $\ralpha$, denoted $B(\ralpha)$,
to be the graph with vertex set equal to $[\ralpha]$, where
$\ralpha,\rbeta\in[\ralpha]$ are connected by an edge if and only if $\ralpha$
and $\rbeta$ are related via a single braid move. Note that we are defining
braid graphs with respect to a fixed reduced expression (or equivalence class)
as opposed to the corresponding group element. The latter are the graphs that
arise from contracting the edges corresponding to braid moves in the Matsumoto
graph.

\begin{example}\label{ex:braid graphs}
Below we describe four different braid classes and illustrate their
corresponding braid graphs.
\begin{enumerate}[(a)] 
\item Consider the Coxeter system of type $A_4$. The braid class for the
reduced expression $1213243$ consists of the following reduced expressions:
\[
    \ralpha_1 := 1213243, \
    \ralpha_2 := 2123243, \
    \ralpha_3 := 2132343, \
    \ralpha_4 := 2132434.
\]
\item In the Coxeter system of type $A_6$, the expression $1213243565$ is
reduced. Its braid class consists of the following reduced expressions:
\[
\rbeta_1 := 1213243565, \
\rbeta_2 := 2123243565, \
\rbeta_3 := 2132343565, \
\rbeta_4 := 2132434565,
\]
\[
\rbeta_5 := 1213243656, \
\rbeta_6 := 2123243656, \
\rbeta_7 := 2132343656, \
\rbeta_8 := 2132434656.
\]
\item Now, consider the Coxeter system of type $D_4$. The expression $2321434$
is reduced and its braid class consists of the following reduced expressions:
\[
\rgamma_1 := 2321434, \
\rgamma_2 := 3231434, \
\rgamma_3 := 2321343, \
\rgamma_4 := 3231343, \
\rgamma_5 := 3213143.
\]

\item Lastly, consider the Coxeter system of type $D_5$. The expressions in
Part~(c) remain reduced. One can extend the braid chain by appending the
letters 54 to the right of every element of $[\rgamma_1]$. The resulting five
expressions are reduced and braid equivalent. However, the expressions
$\rgamma_1 54$ and $\rgamma_2 54$ each provide a new opportunity to apply a
braid move. The resulting braid chain consists of the following 7 elements:
\[
\rdelta_1 := 232143454, \
\rdelta_2 := 323143454, \
\rdelta_3 := 232134354, \
\rdelta_4 := 323134354,
\]
\[
\rdelta_5 := 321314354, \
\rdelta_6 := 232143545, \
\rdelta_7 := 323143545.
\]
\end{enumerate}
The braid graphs $B(\ralpha_1)$, $B(\rbeta_1)$, $B(\rgamma_1)$, and
$B(\rdelta_1)$ are depicted in Figure~\ref{fig:braidgraphs}.
\end{example}

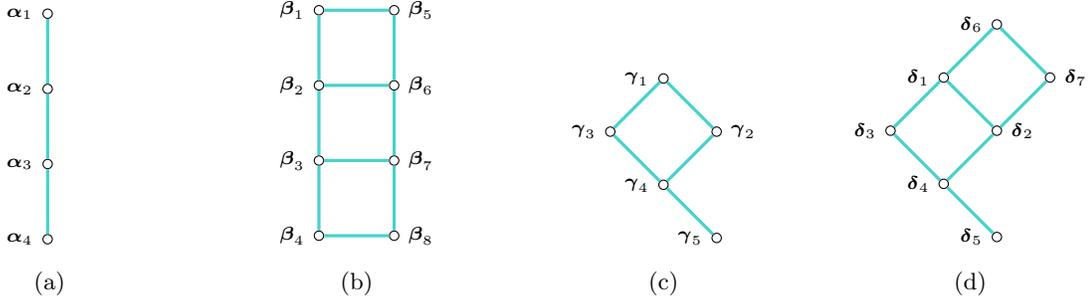
\begin{figure}[!ht]
\centering
\subcaptionbox{\label{fig:braid graph 1}}[.24\linewidth]{
\begin{tikzpicture}[every circle node/.style={draw, circle, inner sep=1.25pt}]
\node [circle] (1) [label=left:$\scriptstyle \ralpha_4$] at (0,1){};
\node [circle] (2) [label=left:$\scriptstyle \ralpha_3$] at (0,2){};
\node [circle] (3) [label=left:$\scriptstyle \ralpha_2$] at (0,3){};
\node [circle] (4) [label=left:$\scriptstyle \ralpha_1$] at (0,4){};
\draw [turq,-, very thick] (1) to (2);
\draw [turq,-, very thick] (2) to (3);
\draw [turq,-, very thick] (3) to (4);
\draw (0,1) node[label=right:\phantom{$\scriptstyle \ralpha_1$}]{};

\end{tikzpicture}
}
\subcaptionbox{\label{fig:braid graph 2}}[.24\linewidth]{
\begin{tikzpicture}[every circle node/.style={draw, circle, inner sep=1.25pt}]
\node [circle] (1) [label=left:$\scriptstyle \rbeta_4$] at (0,1){};
\node [circle] (2) [label=left:$\scriptstyle \rbeta_3$] at (0,2){};
\node [circle] (3) [label=left:$\scriptstyle \rbeta_2$] at (0,3){};
\node [circle] (4) [label=left:$\scriptstyle \rbeta_1$] at (0,4){};
\node [circle] (5) [label=right:$\scriptstyle \rbeta_8$] at (1,1){};
\node [circle] (6) [label=right:$\scriptstyle \rbeta_7$] at (1,2){};
\node [circle] (7) [label=right:$\scriptstyle \rbeta_6$] at (1,3){};
\node [circle] (8) [label=right:$\scriptstyle \rbeta_5$] at (1,4){};

\draw [turq,-, very thick] (1) to (2);
\draw [turq,-, very thick] (2) to (3);
\draw [turq,-, very thick] (3) to (4);
\draw [turq,-, very thick] (5) to (6);
\draw [turq,-, very thick] (6) to (7);
\draw [turq,-, very thick] (7) to (8);
\draw [turq,-, very thick] (1) to (5);
\draw [turq,-, very thick] (2) to (6);
\draw [turq,-, very thick] (3) to (7);
\draw [turq,-, very thick] (4) to (8);
\end{tikzpicture}
}
\subcaptionbox{\label{fig:braidgraphc}}[.24\linewidth]{
\begin{tikzpicture}[every circle node/.style={draw, circle, inner sep=1.25pt}]
\node [circle] (1) [label=left:$\scriptstyle \rgamma_4$] at (0,1){};
\node [circle] (2) [label=left:$\scriptstyle \rgamma_5$] at (.707,.293){};
\node [circle] (3) [label=left:$\scriptstyle \rgamma_3$] at (-.707,1.707){};
\node [circle] (4) [label=right:$\scriptstyle \rgamma_2$] at (.707,1.707){};
\node [circle] (5) [label=left:$\scriptstyle \rgamma_1$] at (0,2.414){};
\draw [turq,-, very thick] (1) to (2);
\draw [turq,-, very thick] (1) to (3);
\draw [turq,-, very thick] (1) to (4);
\draw [turq,-, very thick] (3) to (5);
\draw [turq,-, very thick] (5) to (4);
\end{tikzpicture}
}
\subcaptionbox{\label{fig:braidgraphd}}[.24\linewidth]{
	\begin{tikzpicture}[every circle node/.style={draw, circle, inner sep=1.25pt}]
		\node [circle] (1) [label=left:$\scriptstyle \rdelta_4$] at (0,1){};
		\node [circle] (2) [label=left:$\scriptstyle \rdelta_5$] at (.707,.293){};
		\node [circle] (3) [label=left:$\scriptstyle \rdelta_3$] at (-.707,1.707){};
		\node [circle] (4) [label=right:$\scriptstyle \rdelta_2$] at (.707,1.707){};
		\node [circle] (5) [label=left:$\scriptstyle \rdelta_1$] at (0,2.414){};
		\node [circle] (6) [label=right:$\scriptstyle \rdelta_7$] at (1.414,2.414){};
		\node [circle] (7) [label=left:$\scriptstyle \rdelta_6$] at (.707,3.121){};
		\draw [turq,-, very thick] (1) to (2);
		\draw [turq,-, very thick] (1) to (3);
		\draw [turq,-, very thick] (1) to (4);
		\draw [turq,-, very thick] (3) to (5);
		\draw [turq,-, very thick] (5) to (4);
		\draw [turq,-, very thick] (5) to (7);
		\draw [turq,-, very thick] (6) to (7);
		\draw [turq,-, very thick] (6) to (4);
	\end{tikzpicture}
}
\caption{Braid graphs generated by various reduced
expressions.\label{fig:braidgraphs}}
\end{figure}

\begin{section}{Architecture of braid classes in simply-laced Coxeter
systems}\label{sec:architecture}

In this section, we introduce the notions of braid shadow and link, thus
allowing us to provide a factorization of reduced expressions in simply-laced
Coxeter systems into products of maximal links. This in turn yields a
decomposition of the braid graph for a reduced expression in simply-laced
Coxeter systems into a box product of the braid graphs for the corresponding
link factors.

If $i,j\in\mathbb{N}$ with $i\leq j$, then we define the interval $\llb
i,j\rrb:=\{i,i+1,\ldots, j-1,j\}$. We define the degenerate interval $\llb
i,i\rrb$ to be the singleton set $\{i\}$, which we may write as $\llb i\rrb$.
We will use the intervals $\llb i,j\rrb$ to represent positions in a reduced
expression.

\begin{definition}\label{def:support}
Suppose $(W,S)$ is a Coxeter system. If $\ralpha=s_{x_1}s_{x_2}\cdots s_{x_m}$
is a reduced expression for $w\in W$, we define the \emph{local support} of
$\ralpha$ over $\llb i,j\rrb$ via
\[
\supp_{\llb i,j\rrb}(\ralpha):=\{s_{x_k}\mid k\in \llb i,j\rrb\}. 
\]
The \emph{local support} of the braid class $[\ralpha]$ over $\llb i,j\rrb$ is
defined by
\[
\supp_{\llb i,j\rrb}([\ralpha]):=\bigcup_{\rbeta\in[\ralpha]}\supp_{\llb
i,j\rrb}(\rbeta).
\]
\end{definition}

That is, $\supp_{\llb i,j\rrb}(\ralpha)$ is the set consisting of the
generators that appear in positions $i,i+1,\dots, j$ of $\ralpha$ while
$\supp_{\llb i,j\rrb}([\ralpha])$ is the set of generators that appear in
positions $i,i+1,\dots,j$ of any reduced expression braid equivalent to
$\ralpha$. Also, if $\ralpha = s_{x_1}s_{x_2}\cdots s_{x_m}$, we let
$\ralpha_{\llb i,j\rrb }$ denote the factor $s_{x_i}s_{x_{i+1}}\cdots
s_{x_{j-1}}s_{x_j}$ of $\ralpha$. In the case of the  degenerate interval
$\llbracket i,i\rrbracket$, we will use the notation $\supp_{\llb
i\rrb}(\ralpha)$, $\supp_{\llb i\rrb}([\ralpha])$, and $\ralpha_{\llb i\rrb}$,
and we will simply write $\supp(\ralpha)$ for the set of generators that appear
in $\ralpha$.

Throughout the remainder of this section, we assume that $(W,S)$ is a
simply-laced Coxeter system. It is worth pointing out that many of the
following results do not hold in arbitrary Coxeter systems.

\begin{definition}\label{def:braidshadow}
Suppose that $(W,S)$ is a simply-laced Coxeter system. If $\ralpha =
s_{x_1}s_{x_2}\cdots s_{x_m}$ is a reduced expression for $w\in W$, then the
interval $\llb i,i+2\rrb$ is a \emph{braid shadow} for $\ralpha$ if
$s_{x_i}=s_{x_{i+2}}$ and
$m(s_{x_i},s_{x_{i+1}})=3=m(s_{x_{i+1}},s_{x_{i+2}})$. The collection of braid
shadows for $\ralpha$ is denoted by $\bs(\ralpha)$ and the set of braid shadows
for the braid class $[\ralpha]$ is given by
\[
\bs([\ralpha]) : = \bigcup_{\rbeta\in[\ralpha]}{\bs(\rbeta)}.
\]
The cardinality of $\bs([\ralpha])$ is called the \emph{rank} of $\ralpha$,
which we denote by $\rank(\ralpha)$.
\end{definition}

Note that if $\ralpha = s_{x_1}s_{x_2}\cdots s_{x_m}$ is a reduced expression
for $w\in W$ such that $0\leq m\leq 2$, then $\rank(\ralpha)=0$. Of course, the
converse is not always true.

A braid shadow is simply the location in a reduced expression where we have an
opportunity to apply a braid move. A reduced expression may have many braid
shadows or possibly none at all. The sets $\bs(\ralpha)$ and $\bs([\ralpha])$
capture the locations where braid moves can be performed in $\ralpha$ and any
reduced expression braid equivalent to $\ralpha$, respectively. If $\llb
i,i+2\rrb$ is a braid shadow for $[\ralpha]$, then we may refer to position
$i+1$ in any reduced expression in $[\ralpha]$ as the \emph{center} of the
braid shadow.

\begin{example}\label{ex:braidshadows}
Consider the reduced expressions given in Example~\ref{ex:braid graphs}. By
inspection, we see that:
\begin{enumerate}[(a)]
\item  $\bs(\ralpha_1) = \{\llb 1,3\rrb\}$ and $\bs([\ralpha_1]) = \{\llb
1,3\rrb,\llb 3,5\rrb, \llb 5,7\rrb\}$,
\item $\bs(\rbeta_1) = \{\llb 1,3\rrb ,\llb 8,10\rrb\}$ and $\bs([\rbeta_1]) =
\{\llb 1,3\rrb,\llb 3,5\rrb,\llb 5,7\rrb,\llb 8,10\rrb\}$,
\item $\bs(\rgamma_1) = \{\llb 1,3\rrb,\llb 5,7\rrb\}$ and $\bs([\rgamma_1]) =
\{\llb 1,3\rrb,\llb 3,5\rrb,\llb 5,7\rrb\}$.
\end{enumerate}
\end{example}

The following metric will be useful in the proof of
Proposition~\ref{prop:hugh}.

\begin{definition}\label{def:metric}
Suppose $(W,S)$ is a simply-laced Coxeter system. If $\ralpha$ and $\rbeta$ are
two braid equivalent reduced expressions for $w\in W$, then the \emph{braid
distance} $d(\ralpha,\rbeta)$ between $\ralpha$ and $\rbeta$ is defined to be
the minimum number of braid moves required to transform $\ralpha$ into $\rbeta$.
\end{definition}

Equivalently, we can interpret the number $d(\ralpha,\rbeta)$ as the length of
any minimal path from $\ralpha$ to $\rbeta$ in the corresponding braid graph.
This is a consequence of Matsumoto's Theorem
(Proposition~\ref{prop:matsumoto}).

Section~2.1 of~\cite{Fishel2018} explicitly states that for Coxeter systems of
type $A_n$, braid shadows for a braid class are either disjoint or overlap by a
single position. The next proposition extends this phenomenon to arbitrary
simply-laced Coxeter systems.

\begin{proposition}\label{prop:hugh}
Suppose $(W,S)$ is a simply-laced Coxeter system. If $\ralpha$ is a reduced
expression for $w\in W$ with $\llb i,i+2\rrb \in \bs([\ralpha])$, then $\llb
i+1,i+3\rrb \not \in \bs([\ralpha])$.
\end{proposition}

\begin{proof}
The result is trivial if $\ell(w)\leq 3$. Suppose $\ell(w)=m \geq 4$. For each
$i\in \{1,2,\ldots, m-3\}$, define
\[
P_i := \big\{ (\rbeta,\rgamma) \mid \rbeta, \rgamma\in[\ralpha], \llb
i,i+2\rrb\in \bs(\rbeta) \ \text{and} \ \llb i+1,i+3\rrb\in\bs(\rgamma) \big\}
\]
and let
\[
P := \bigcup_{i=1}^{m-3}{P_i}.
\]
We will prove that $P=\emptyset$. Assume otherwise and choose $(\rbeta,\rgamma)
\in P$ such that $d(\rbeta,\rgamma)$ is minimal among all elements of $P$.
Since $(\rbeta,\rgamma)\in P$, there exists $i\in \{1,\ldots,m-3\}$ such that
$(\rbeta,\rgamma)\in P_i$. This implies that $\llb i,i+2\rrb\in \bs(\rbeta)$
while $\llb i+1,i+3\rrb\in \bs(\rgamma)$. Let $\rbeta_{\llb i,i+1\rrb}=sts$ and
$\rgamma_{\llb i+1,i+3\rrb}=uvu$, where $m(s,t)=3=m(u,v)$. Suppose
$d(\rbeta,\rgamma) = k$ and say $\ralpha_0:=\rbeta, \ralpha_1, \ldots,
\ralpha_{k-1}, \ralpha_k:=\rgamma$ is a minimal sequence of braid equivalent
reduced expressions, each one braid move apart, that transforms $\rbeta$ into
$\rgamma$ in $k$ braid moves. Let $b_j$ denote the braid move that transforms
$\ralpha_{j-1}$ into $\ralpha_j$. We can represent this sequence of moves
visually as follows:
\[
\underbrace{\subword{s}{i+1}{t}{i+2}{s}{i+3}{?}}_{\ralpha_0}
\ \overset{b_1}{\longmapsto}\cdots \overset{b_k}{\longmapsto} \
\underbrace{\subword{?}{i+1}{u}{i+2}{v}{i+3}{u}}_{\ralpha_k}
\]
Suppose $b_1$ transforms $\ralpha_0$ into $\ralpha_1$ such that $b_1$ does not
involve position $i$ nor position $i+3$. Then $(\ralpha_1,\ralpha_k)\in P_{i}$
while $d(\ralpha_1,\ralpha_k) = k-1$, a contradiction. So, the opening braid
move $b_1$ must involve either position $i$ or position $i+3$. However, if
$b_1$ acts on positions $\llb i-1,i+1\rrb$ or $\llb i+1,i+3\rrb$, then
$\ralpha_0$ is not reduced. On the other hand, if $b_1$ transforms $\ralpha_0$
into $\ralpha_1$ by applying the braid move $sts \mapsto tst$ in positions
$\llb i,i+2\rrb$, then $(\ralpha_1,\ralpha_k) \in P_i$ while
$d(\ralpha_1,\ralpha_k) = k-1$, again a contradiction. Hence $b_1$ must act on
positions $\llb i+2,i+4\rrb$ or positions $\llb i-2,i\rrb$.

Assume that $b_1$ acts on positions $\llb i+2,i+4\rrb$. Then there exists $x\in
S$ with $m(s,x) = 3$ such that
\[
\underbrace{\subword{s}{i+1}{t}{i+2}{s}{i+3}{x}{i+4}{s}}_{\ralpha_0}
\ \overset{b_1}{\longmapsto} \
\underbrace{\subword{s}{i+1}{t}{i+2}{x}{i+3}{s}{i+4}{x}}_{\ralpha_1}
\]
This implies that $(\ralpha_k,\ralpha_1) \in P_{i+1}$ while
$d(\ralpha_k,\ralpha_1) = k-1$, which is a contradiction. We can conclude that
$b_1$ acts on positions $\llb i-2,i\rrb$. Now, define the subsequences
\[
L : = \{b_n \mid b_n \text{ acts on } \llb j,j+2\rrb \ \text{for} \ j< i-1\}
\text{ and } R: = \{b_n \mid b_n \text{ acts on } \llb j,j+2\rrb \ \text{for} \
j\geq i-1\}.
\]
Note that $b_1\in L$. We will show that $R=\emptyset$. Assume otherwise and let
$b_r\in R$. If $b_r$ acts on $\llb i-1,i+1\rrb$, then
$(\ralpha_{r-1},\ralpha_0)\in P_{i-1}$ while $d(\ralpha_{r-1},\ralpha_0) = r-1
< k$. Similarly, if $b_r$ acts on $\llb i,i+2\rrb$, then
$(\ralpha_{r-1},\ralpha_k)\in P_i$ while $d(\ralpha_{r-1},\ralpha_k) = k-(r-1)
< k$. In either case, we contradict the minimality of $k$. This shows that the
positions that the braid moves in $L$ and $R$ act on respectively do not
overlap, and so the braid moves in $R$ can be applied in any order relative to
the braid moves in $L$. In particular, the braid moves in $R$ could be applied
prior to any of the braid moves in $L$. But this contradicts the fact that
$b_1\in L$. Therefore, the sequence $b_1,\ldots,b_k$ never acts on positions to
the right of position $i$. But this is impossible since these positions are
disjoint from $\llb i+1,i+3\rrb$, which we must necessarily change to arrive at
$\ralpha_k$. We conclude that $P$ is empty, which yields the desired result.
\end{proof}

Proposition~\ref{prop:hugh} motivates the following definition.

\begin{definition}\label{def:braidchain}
Suppose $(W,S)$ is a simply-laced Coxeter system. If $\ralpha =
s_{x_1}s_{x_2}\cdots s_{x_m}$ is a reduced expression for $w\in W$ with $m\geq
1$, then we say that $\ralpha$ is a \emph{link} provided that either $m=1$ or
$m$ is odd and $\bs([\ralpha]) = \{\llb 1,3\rrb,\llb 3,5\rrb,\ldots,\llb
m-4,m-2\rrb,\llb m-2,m\rrb\}$. If $\ralpha$ is a link, then the corresponding
braid class $[\ralpha]$ is called a \emph{braid chain}.
\end{definition}

Loosely speaking, $\ralpha$ is link if there is a sequence of overlapping braid
moves that ``cover" the positions $1,2,\ldots, m$. Note that if $\ralpha$ is a
link, then the rank of $\ralpha$ is $k$ if and only if $\ralpha$ consists of
$2k+1$ letters. Notice that the center of every braid shadow for a braid chain
is an even index.

\begin{example}\label{ex:linkschains}
Consider the reduced expressions given in Example~\ref{ex:braid graphs}. Since
$\bs([\ralpha_1]) = \{\llb 1,3\rrb,\llb 3,5\rrb,\llb 5,7\rrb\}$, $\ralpha_1$ is
a link and $[\ralpha_1]$ is a braid chain. On the other hand, since
$\bs([\rbeta_1]) = \{\llb 1,3\rrb,\llb 3,5\rrb,\llb 5,7\rrb,\llb 8,10\rrb\}$,
it follows that $\rbeta_1$ is not a link and hence $[\rbeta_1]$ is not a braid
chain. However, it turns out that the factors $1213243$ and $565$ of $\rbeta_1$
are links in their own right. Lastly, since $\bs([\rgamma_1])=\{\llb
1,3\rrb,\llb 3,5\rrb,\llb 5,7\rrb\}$, $\rgamma_1$ is a link and $[\rgamma_1]$
is a braid chain.
\end{example}

\begin{definition}\label{def:link factor}
Suppose $(W,S)$ is a simply-laced Coxeter system. If $\ralpha$ is a reduced
expression for $w\in W$ with $\ell(w)\geq 1$, then we say that $\rbeta$ is a
\emph{link factor} of $\ralpha$ provided that
\begin{enumerate}[(a)]
\item $\rbeta$ is a factor of $\ralpha$,
\item $\rbeta$ is a link, and
\item for every factor $\rgamma$ of $\ralpha$, if $\rbeta$ is a factor of
$\rgamma$ and $\rgamma$ is a link, then $\rbeta=\rgamma$.
\end{enumerate}
\end{definition}

It follows immediately from Definition~\ref{def:link factor} that every reduced
expression $\ralpha$ for a nonidentity group element can be written uniquely as
a product of link factors, say $\ralpha_1\ralpha_2\cdots\ralpha_k$, where each
$\ralpha_i$ is a link factor of $\ralpha$. We refer to this product as the
\emph{link factorization} of $\ralpha$. For emphasis, we will often denote such
a factorization via $\ralpha = \ralpha_1\mid\ralpha_2\mid\cdots\mid \ralpha_k$.
For convenience we say that the link factorization of the identity is product
of a single copy of the empty word, but it is important to note that the empty
word is not actually a link. The following proposition is a direct consequence
of the definitions.

\begin{proposition}\label{prop:class as products of factors}
Suppose $(W,S)$ is a simply-laced Coxeter system. If $\ralpha$ is a reduced
expression for $w\in W$ with link factorization
$\ralpha_1\mid\ralpha_2\mid\cdots\mid\ralpha_k$, then
\[
[\ralpha] = \big\{\rbeta_1\mid \rbeta_2\mid \cdots\mid \rbeta_k : \rbeta_i \in
[\ralpha_i]\text{ for } 1\leq i\leq k\big\}.
\]
Moreover, the cardinality of the braid class for $\ralpha$ is given by
\[
\card([\ralpha]) = \prod_{i=1}^k\card([\ralpha_i]),
\]
and the rank of $\ralpha$ is given by
\[
\rank(\ralpha) = \sum_{i=1}^k \rank(\ralpha_i).
\]
\end{proposition}

Proposition~\ref{prop:class as products of factors} implies that the braid
graph for any reduced expression for a group element can be decomposed as the
box product of the braid graphs for the corresponding link factors in the link
factorization. Note that the decomposition is unique if one respects the
ordering of the link factors.

\begin{corollary}\label{cor:boxproductsofbraidgraphs}
Suppose $(W,S)$ is a simply-laced Coxeter system. If $\ralpha$ is a reduced
expression for $w\in W$ with link factorization
$\ralpha_1\mid\ralpha_2\mid\cdots\mid\ralpha_k$, then $B(\ralpha) \cong
B(\ralpha_1) \square B(\ralpha_2) \square \cdots \square B(\ralpha_k)$.
\end{corollary}

\begin{proof}
An isomorphism of graphs is given by $\rbeta_1\mid \rbeta_2\mid \cdots\mid
\rbeta_k\mapsto (\rbeta_1,\rbeta_2,\ldots,\rbeta_k)$, where each $\rbeta_j\in
[\ralpha_j]$. This bijection between vertex sets respects the edges of the
corresponding graphs since braid moves on distinct link factors can be applied
independently.
\end{proof}

\begin{example}\label{ex:linkfactorization}
Consider the reduced expression $\rbeta_1= 1213243565$ defined in
Example~\ref{ex:braid graphs}. The link factorization for $\rbeta_1$ is
$\textcolor{turq}{1213243} \mid \textcolor{magenta}{565}$. The decomposition
$B(\rbeta_1) \cong B(\textcolor{turq}{1213243}) \square
B(\textcolor{magenta}{565})$ guaranteed by
Corollary~\ref{cor:boxproductsofbraidgraphs} is shown in
Figure~\ref{fig:boxproductofbraidgraphs}. Note that we have utilized colors to
help distinguish the link factors.
\end{example}

\begin{figure}[t!]
\centering
\begin{tikzpicture}[every circle node/.style={draw, circle, inner sep=1.25pt}]
\node [circle] (node_1) at (0,0)  {};
\node [below left,rotate=45] (label_1) at (0,0) {$\scriptstyle
\textcolor{turq}{2123243}\,\mid\,\textcolor{magenta}{656}$};
\node [circle] (node_2) at (1,0) {};
\node [below left,rotate=45] (label_2) at (1,0) {$\scriptstyle
\textcolor{turq}{2132343}\,\mid\,\textcolor{magenta}{656}$};
\node [circle]  (node_3) at (-1,0) {};
\node [below left,rotate=45] (label_3) at (-1,0) {$\scriptstyle
\textcolor{turq}{1213243}\,\mid\,\textcolor{magenta}{656}$};
\node [circle]  (node_4) at (0,1) {};
\node [above right,rotate=45] (label_4) at (0,1) {$\scriptstyle
\textcolor{turq}{2123243}\,\mid\,\textcolor{magenta}{565}$};
\node [circle]  (node_5) at (1,1) {};
\node [above right,rotate=45] (label_5) at (1,1) {$\scriptstyle
\textcolor{turq}{2132343}\,\mid\,\textcolor{magenta}{565}$};
\node [circle]  (node_6) at (-1,1) {};
\node [above right,rotate=45] (label_6) at (-1,1) {$\scriptstyle
\textcolor{turq}{1213243}\,\mid\,\textcolor{magenta}{565}$};
\node [circle]  (node_7) at (2,0) {};
\node [below left,rotate=45] (label_8) at (2,0) {$\scriptstyle
\textcolor{turq}{2132434}\,\mid\,\textcolor{magenta}{656}$};
\node [circle]  (node_8) at (2,1) {};
\node [above right,rotate=45] (label_8) at (2,1) {$\scriptstyle
\textcolor{turq}{2132434}\,\mid\,\textcolor{magenta}{565}$};
\draw [turq,-, very thick] (node_1) to (node_3);
\draw [turq,-, very thick] (node_1) to (node_2);
\draw [magenta,-, very thick] (node_1) to (node_4);
\draw [magenta,-, very thick] (node_3) to (node_6);
\draw [magenta,-, very thick] (node_2) to (node_5);
\draw [turq,-, very thick] (node_4) to (node_5);

\draw [turq,-, very thick] (node_4) to (node_6);
\draw [turq,-, very thick] (node_2) to (node_7);
\draw [turq,-, very thick] (node_5) to (node_8);
\draw [magenta,-, very thick] (node_7) to (node_8);

\node (e) at (2.75,.5) {$\cong$};

\node [circle] (1) at (3.5,2){};
\node [right] (l1) at (3.5,2){$\scriptstyle \textcolor{turq}{2132434}$};
\node [right] (l2) at (3.5,1){$\scriptstyle \textcolor{turq}{2132343}$};
\node [circle] (2) at (3.5,1){};
\node [circle] (3) at (3.5,0){};
\node [right] (l3) at (3.5,0){$\scriptstyle \textcolor{turq}{2123243}$};
\node [circle] (4) at (3.5,-1){};
\node [right] (l4) at (3.5,-1){$\scriptstyle \textcolor{turq}{1213243}$};

\draw [-, very thick, turq] (1) to (2);
\draw [-, very thick, turq] (2) to (3);
\draw [-, very thick, turq] (3) to (4);

\node (x) at (5,.5) {$\Box$};

\node [circle] (5) at (5.5,1){};
\node [circle] (6) at (5.5,0){};
\node [right] (l5) at (5.5,1){$\scriptstyle \textcolor{magenta}{565}$};
\node [right] (l6) at (5.5,0){$\scriptstyle \textcolor{magenta}{656}$};
\draw [-, very thick, magenta] (5) to (6);

\end{tikzpicture}
\caption{Braid graph for the reduced expression from
Example~\ref{ex:linkfactorization} and its decomposition into a box product of
braid graphs for the corresponding link
factors.}\label{fig:boxproductofbraidgraphs}
\end{figure}
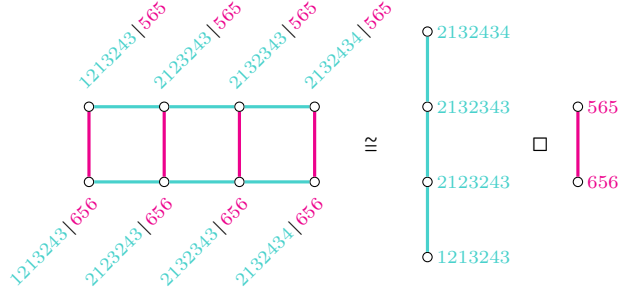

\begin{example}\label{ex:productoflollipops}
Consider the Coxeter system of type $D_7$ determined by the Coxeter graph in
Figure~\ref{fig:labeledB}. The reduced expression $3231343567543231343$ has
link factorization
\[
\textcolor{turq}{3231343}\mid 5\mid 6\mid 7\mid 5\mid 4\mid
\textcolor{magenta}{3231343}.
\]
The braid graphs for the first and last link factors are isomorphic to the
braid graph in Figure~\ref{fig:braidgraphc}. The braid graph for each singleton
factor consists of a single vertex. The braid graph for the entire reduced
expression and its decomposition are shown in
Figure~\ref{fig:productoflollipops}.
\end{example}

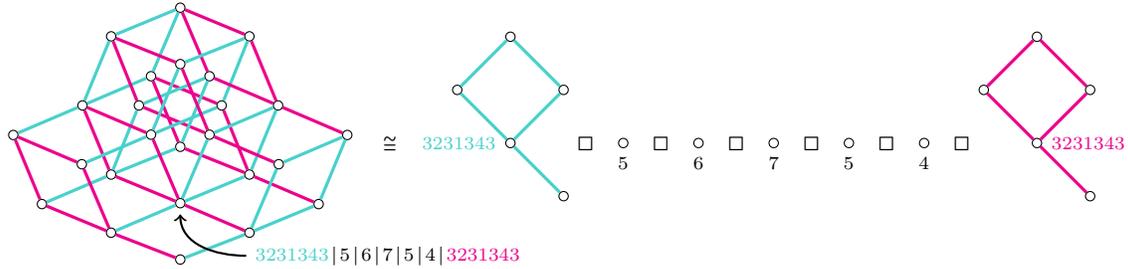
\begin{figure}[ht!]
\begin{tikzpicture}[every circle node/.style={draw, circle, inner sep=1.25pt}]
\node [circle] (1) at (1.3,0){};
\node [circle] (2) at ({1.3*cos(45)},{1.3*sin(45)}){};
\node [circle] (3) at (0,1.3){};
\node [circle] (4) at ({1.3*cos(135)},{1.3*sin(135)}){};
\node [circle] (5) at (-1.3,0){};
\node [circle] (6) at ({1.3*cos(225)},{1.3*sin(225)}){};
\node [circle] (7) at (0,-1.3){};
\node [circle] (8) at ({1.3*cos(315)},{1.3*sin(315)}){};
\draw [magenta,-, very thick] (1) to (2);
\draw [magenta,-, very thick] (2) to (3);
\draw [turq,-, very thick] (3) to (4);
\draw [turq,-, very thick] (4) to (5);
\draw [magenta,-, very thick] (5) to (6);
\draw [magenta,-, very thick] (6) to (7);
\draw [turq,-, very thick] (7) to (8);
\draw [turq,-, very thick] (8) to (1);

\node [circle] (9) at (.55,0){};
\node [circle] (10) at ({.55*cos(45)},{.55*sin(45)}){};
\node [circle] (11) at (0,.55){};
\node [circle] (12) at ({.55*cos(135)},{.55*sin(135)}){};
\node [circle] (13) at (-.55,0){};
\node [circle] (14) at ({.55*cos(225)},{.55*sin(225)}){};
\node [circle] (15) at (0,-.55){};
\node [circle] (16) at ({.55*cos(315)},{.55*sin(315)}){};

\draw [magenta,-, very thick] (8) to (9);

\draw [turq,-, very thick] (1) to (16);
\draw [turq,-, very thick] (2) to (9);
\draw [magenta,-, very thick] (3) to (10);

\draw [turq,-, very thick] (5) to (12);
\draw [turq,-, very thick] (6) to (13);

\draw [magenta,-, very thick] (8) to (15);
\draw [magenta,-, very thick] (1) to (10);
\draw [turq,-, very thick] (2) to (11);
\draw [turq,-, very thick] (3) to (12);
\draw [magenta,-, very thick] (4) to (13);
\draw [magenta,-, very thick] (5) to (14);
\draw [turq,-, very thick] (6) to (15);
\draw [turq,-, very thick] (7) to (16);

\draw [magenta,-, very thick] (9) to (12);
\draw [turq,-, very thick] (10) to (13);

\draw [magenta,-, very thick] (12) to (15);
\draw [magenta,-, very thick] (13) to (16);
\draw [turq,-, very thick] (14) to (9);
\draw [turq,-, very thick] (15) to (10);
\draw [magenta,-, very thick] (16) to (11);

\node [circle] (17) at (1.3+.92,-.393){};
\node [circle] (18) at (0+.92,-1.3-.393){};
\node [circle] (19) at ({1.3*cos(315)+.92},{1.3*sin(315)-.393}){};
\node [circle] (20) at ({.55*cos(315)+.92},{.55*sin(315)-.393}){};

\node [circle] (21) at (-1.3-.92,0-.393){};
\node [circle] (22) at ({1.3*cos(225)-.92},{1.3*sin(225)-.393}){};
\node [circle] (23) at (0-.92,-1.3-.393){};
\node [circle] (24) at ({.55*cos(225)-.92},{.55*sin(225)-.393}){};

\draw [turq,-, very thick] (22) to (6);
\draw [turq,-, very thick] (21) to (5);
\draw [turq,-, very thick] (23) to (7);
\draw [turq,-, very thick] (24) to (14);

\draw [magenta,-, very thick] (19) to (8);
\draw [magenta,-, very thick] (17) to (1);
\draw [magenta,-, very thick] (18) to (7);
\draw [magenta,-, very thick] (20) to (16);

\draw [turq,-, very thick] (17) to (19);
\draw [turq,-, very thick] (18) to (19);
\draw [turq,-, very thick] (17) to (20);
\draw [turq,-, very thick] (20) to (18);

\draw [magenta,-, very thick] (21) to (22);
\draw [magenta,-, very thick] (22) to (23);
\draw [magenta,-, very thick] (23) to (24);
\draw [magenta,-, very thick] (24) to (21);

\draw [magenta,-, very thick] (7) to (14);

\draw [turq,-, very thick] (11) to (14);
\draw [magenta,-, very thick] (4) to (11);

\node [circle] (25) at (0,-2.05){};
\draw [turq,-, very thick] (18) to (25);
\draw [magenta,-, very thick] (23) to (25);

\node (a) at (1.3+1.089+.4,-.504) {$\cong$};

\node [circle] (28) [label=left: $\scriptstyle \textcolor{turq}{3231343}$] at
(1.3+1.089+2,1-1.5){};
\node [circle] (29) at (.707+1.3+1.089+2,.293-1.5){};
\node [circle] (30) at (-.707+1.3+1.089+2,1.707-1.5){};
\node [circle] (31) at (.707+1.3+1.089+2,1.707-1.5){};
\node [circle] (32) at (0+1.3+1.089+2,2.414-1.5){};
\draw [turq,-, very thick] (30) to (28);
\draw [turq,-, very thick] (28) to (29);
\draw [turq,-, very thick] (28) to (31);
\draw [turq,-, very thick] (30) to (32);
\draw [turq,-, very thick] (32) to (31);

\node (b) at (1.3+1.089+3,-.5) {$\Box$};

\node [circle] (26) [label=below:$\scriptstyle 5$] at (1.3+1.089+3.5,-.5){};

\node (c) at (1.3+1.089+4,-.5) {$\Box$};

\node [circle] (27) [label=below:$\scriptstyle 6$] at (1.3+1.089+4.5,-.5){};

\node (d)  at (1.3+1.089+5,-.5) {$\Box$};

\node [circle] (28) [label=below:$\scriptstyle 7$]  at (1.3+1.089+5.5,-.5){};

\node (e)  at (1.3+1.089+6,-.5) {$\Box$};

\node [circle] (29) [label=below:$\scriptstyle 5$] at (1.3+1.089+6.5,-.5){};

\node (f) at (1.3+1.089+7,-.5) {$\Box$};

\node [circle] (30) [label=below:$\scriptstyle 4$] at (1.3+1.089+7.5,-.5){};

\node (g) at (1.3+1.089+8,-.5) {$\Box$};

\node [circle] (33) [label=right: $\scriptstyle \textcolor{magenta}{3231343}$]
at (0+1.3+1.089+9,1-1.5){};
\node [circle] (34) at (.707+1.3+1.089+9,.293-1.5){};
\node [circle] (35) at (-.707+1.3+1.089+9,1.707-1.5){};
\node [circle] (36) at (.707+1.3+1.089+9,1.707-1.5){};
\node [circle] (37) at (0+1.3+1.089+9,2.414-1.5){};
\draw [magenta,-, very thick] (33) to (34);
\draw [magenta,-, very thick] (33) to (35);
\draw [magenta,-, very thick] (33) to (36);
\draw [magenta,-, very thick] (35) to (37);
\draw [magenta,-, very thick] (37) to (36);
\node (h) [label = right: $\scriptstyle \textcolor{turq}{3231343} \, \mid \, 5
\, \mid \, 6 \, \mid \, 7 \, \mid \, 5 \, \mid \, 4 \, \mid \,
\textcolor{magenta}{3231343} $] at (.75,-2){};
\node (y) at (0,-1.325){};
\node (x) at (1,-2){};
\draw [->, thick] (x) to [out=180,in=270] (y);
\end{tikzpicture}
\caption{Braid graph for the reduced expression from
Example~\ref{ex:productoflollipops} and its decomposition into a box product of
braid graphs for the corresponding link factors.}\label{fig:productoflollipops}
\end{figure}

According to the next proposition, the support of a braid shadow is constant
across an entire braid class in simply-laced triangle-free Coxeter systems.

\begin{proposition}\label{prop:equalsupport}
Suppose $(W,S)$ is a simply-laced triangle-free Coxeter system. If $\ralpha$
and $\rbeta$ are two braid equivalent reduced expressions for $w\in W$ with
$\ell(w)\geq 3$, then for all $\llb i,i+2\rrb \in \bs(\ralpha)\cap
\bs(\rbeta)$, $\supp_{\llb i,i+2\rrb}(\ralpha) = \supp_{\llb
i,i+2\rrb}(\rbeta)$.
\end{proposition}

\begin{proof}
Let $\ralpha$ and $\rbeta$ be two braid equivalent reduced expressions for
$w\in W$ with $\ell(w)=m\geq 3$. For each $i\in \{1,\ldots, m-2\}$, define
\[
P_i:=\big\{(\rgamma,\rdelta)  \mid \rgamma,\rdelta \in [\ralpha], \llb
i,i+2\rrb\in \bs(\rgamma)\cap \bs(\rdelta), \ \text{and} \ \supp_{\llb
i,i+2\rrb}(\rgamma) \neq \supp_{\llb i,i+2\rrb}(\rdelta)\big\},
\]
and let
\[
P:=\bigcup_{i=1}^{m-2}P_i.
\]
We will prove that $P = \emptyset$. Suppose otherwise and choose
$(\rgamma,\rdelta)\in P$ such that $d(\rgamma,\rdelta)$ is minimal among all
elements of $P$. Then there exists $i\in\{1,\ldots, m-2\}$ such that
$(\rgamma,\rdelta)\in P_i$, so that $\llb i,i+2\rrb\in \bs(\rgamma)\cap
\bs(\rdelta)$ while $\supp_{\llb i,i+2\rrb}(\rgamma) \neq \supp_{\llb
i,i+2\rrb}(\rdelta)$.
Let $\rgamma_{\llb i,i+2\rrb}=sts$ and $\rdelta_{\llb i,i+2\rrb}=uvu$, where
$m(s,t)=3=m(u,v)$. Suppose $d(\rgamma,\rdelta) = k$ and let
\[
\ralpha_0:=\rgamma, \ralpha_1, \ldots, \ralpha_{k-1}, \ralpha_k:=\rdelta
\]
be a minimal sequence of braid equivalent reduced expressions that transforms
$\rgamma$ into $\rdelta$ in $k$ braid moves such that
$d(\ralpha_{j-1},\ralpha_{j})=1$ for $1\leq j\leq k$. It is clear that $k\geq
2$. Let $b_j$ denote the braid move that transforms $\ralpha_{j-1}$ into
$\ralpha_j$. As in the proof of Proposition~\ref{prop:hugh}, we represent this
sequence of moves visually as follows:
\[
\underbrace{\subword{s}{i+1}{t}{i+2}{s}}_{\ralpha_0} 
\  \overset{b_1}{\longmapsto}\cdots \overset{b_k}{\longmapsto}  \ 
\underbrace{\subword{u}{i+1}{v}{i+2}{u}}_{\ralpha_k}
\]
Since $\llb i,i+2\rrb\in\bs(\ralpha_0)$, the intervals $\llb i-1,i+1\rrb$ and
$\llb i+1,i+3\rrb$ are not braid shadows for every reduced expression in the
sequence $\ralpha_0, \ralpha_1, \ldots, \ralpha_{k-1}, \ralpha_k$ by
Proposition~\ref{prop:hugh}. By the minimality of $k$, the interval $\llb
i,i+2\rrb$ only appears as a braid shadow in $\ralpha_0$ and $\ralpha_k$. That
is, $\llb i,i+2\rrb\notin \bs(\ralpha_l)$ for all $1\leq l\leq k-1$.  Together
these facts imply that $\supp_{\llb i+1\rrb}(\ralpha_l) = \{t\}$ for all $0\leq
l\leq k$. In other words, $t$ is fixed in position $i+1$ throughout the entire
sequence $\ralpha_0, \ralpha_1, \ldots, \ralpha_{k-1}, \ralpha_k$. This forces
$v=t$, which in turn implies that $m(u,t) = 3$. Again by the minimality of $k$,
it must be the case that $b_1$ acts on either $\llb i-2,i\rrb$ or $\llb i+2,i+4
\rrb$. Without loss of generality, assume that $b_1$ acts on $\llb i+2,i+4
\rrb$. Then there exists $x \in S$ with $m(s,x) = 3$ such that $\supp_{\llb
i+3\rrb}(\ralpha_0) = \{x\}$ and $\supp_{\llb i+4\rrb}(\ralpha_0) = \{s\}$. In
summary, we have
\[
\underbrace{\subword{s}{i+1}{t}{i+2}{s}{i+3}{x}{i+4}{s}}_{\ralpha_{0}}
\ \overset{b_1}{\longmapsto} \
\underbrace{\subword{s}{i+1}{t}{i+2}{x}{i+3}{s}{i+4}{x}}_{\ralpha_1}
\ \overset{b_{2}}{\longmapsto} \ \cdots \ \overset{b_k}{\longmapsto} \
\underbrace{\subword{u}{i+1}{t}{i+2}{u}{i+3}{?}{i+4}{?}}_{\ralpha_k}
\] 
Towards a contradiction, suppose $x\neq u$. In order to exchange $x$ with $u$
in position $i+2$, there must exist a reduced expression $\ralpha_j$ with $2
\leq j \leq k$ such that $\llb i+2,i+4\rrb \in \bs(\ralpha_j)$ and $\supp_{\llb
i+2,i+4\rrb}(\ralpha_j) \neq \{x,s\}$. Yet if $\supp_{\llb
i+2,i+4\rrb}(\ralpha_j) \neq \{x,s\}$, then $(\ralpha_1,\ralpha_j) \in P_{i+2}$
while $d(\ralpha_1,\ralpha_j) = j-1 < k$, a contradiction. Hence $x=u$. This
implies that $m(s,u) = 3$, and hence $m(s,u) = m(u,t) = m(t,s)=3$. But this is
contrary to the fact that $(W,S)$ is triangle free. We conclude that
$P=\emptyset$, which proves the claim.
\end{proof} 

As the next example illustrates, the previous result is false without the
assumption that the Coxeter system is triangle free.

\begin{example}\label{ex:threecycle}
Consider the Coxeter system of type $\widetilde{A}_2$, which is determined by
the Coxeter graph in Figure~\ref{fig:labeledC}. The expression $\ralpha =
1213121$ is a link, and it is easy to see that $\rbeta = 2123212 \in
[\ralpha]$. However, $\supp_{\llb 3,5\rrb}(\ralpha) = \{1,3\}$ while
$\supp_{\llb 3,5\rrb}(\rbeta) = \{2,3\}$. This shows that
Proposition~\ref{prop:equalsupport} is false when the Coxeter graph has a
three-cycle.
\end{example}

If one generalizes the notions of braid shadow and link in the natural way, we
conjecture that a result analogous to Proposition~\ref{prop:equalsupport} holds
in arbitrary Coxeter systems as long as the corresponding Coxeter graph does
not contain a three-cycle with edge weights $3, 3, m$, where $m\geq 3$.

When a reduced expression has a braid shadow, the collection of generators that
may appear at the center of the braid shadow in any braid equivalent reduced
expression is completely determined by the support of that braid shadow. The
following proposition makes this more precise.

\begin{proposition}\label{prop:core}
Suppose $(W,S)$ is a simply-laced triangle-free Coxeter system and let
$\ralpha$ be a link of rank at least one. Then $\llb 2i-1,2i+1\rrb\in
\bs(\ralpha)$ if and only if $\supp_{\llb 2i-1,2i+1\rrb}(\ralpha) =\supp_{\llb
2i\rrb}([\ralpha])$.
\end{proposition}

\begin{proof}
To prove the forward implication, let $\llb 2i-1,2i+1\rrb\in \bs(\ralpha)$ and
assume that $\supp_{\llb 2i-1,2i+1\rrb}(\ralpha) = \{s,t\}$. It is clear that
$\{s,t\}\subseteq \supp_{\llb 2i\rrb}([\ralpha])$. We may assume without loss
of generality that $\supp_{\llb 2i\rrb}(\ralpha) = \{s\}$. Let $u\in
\supp_{\llb 2i\rrb}([\ralpha])$ and choose $\rbeta\in [\ralpha]$ such that
$\supp_{\llb 2i\rrb}(\rbeta) = \{u\}$. Then $\ralpha$ and $\rbeta$ are related
by a sequence of braid moves. If no braid move involves position $2i$, then
$s=u$. Otherwise, the sequence has a braid move that involves position $2i$.
However, by Proposition~\ref{prop:hugh}, $\llb 2i-2,2i\rrb, \llb
2i,2i+2\rrb\notin \bs([\ralpha])$. Hence there exists $\rgamma\in[\ralpha]$
such that $\llb 2i-1,2i+1\rrb\in \bs(\rgamma)$ and $u \in \supp_{\llb
2i-1,2i+1\rrb}(\rgamma)$. By Proposition~\ref{prop:equalsupport}, $\supp_{\llb
2i-1,2i+1\rrb}(\rgamma) = \supp_{\llb 2i-1,2i+1\rrb}(\ralpha)$. This shows that
$u \in\{s,t\}$, and so $\supp_{\llb 2i\rrb}([\ralpha])=\{s,t\}$, as desired.

To prove the converse, assume $\supp_{\llb 2i-1,2i+1\rrb}(\ralpha) =\supp_{\llb
2i\rrb}([\ralpha])$. Since $\ralpha$ is a link, we know $\llb 2i-1,2i+1\rrb\in
\bs([\ralpha])$, and hence we can choose $\rbeta \in [\ralpha]$ such that $\llb
2i-1,2i+1\rrb\in \bs(\rbeta)$. But now we can apply the forward implication to
$\rbeta$ to conclude that
\[
\supp_{\llb 2i-1,2i+1\rrb}(\rbeta) =\supp_{\llb 2i\rrb}([\rbeta])=\supp_{\llb
2i\rrb}([\ralpha])=\supp_{\llb 2i-1,2i+1\rrb}(\ralpha).
\]
This implies that $\llb 2i-1,2i+1 \rrb \in \bs(\ralpha)$.
\end{proof}

Applying the previous proposition to a pair of overlapping braid shadows yields
the following corollary, which says that the supports of overlapping braid
shadows intersect at a single element.

\begin{corollary}\label{cor:singletonintersection}
Suppose $(W,S)$ is a simply-laced triangle-free Coxeter system. If $\ralpha$ is
a link of rank at least two, then $\card(\supp_{\llb
2i\rrb}([\ralpha])\cap\supp_{\llb 2i+2\rrb}([\ralpha])) = 1$.
\end{corollary}

If $\ralpha$ is a link while $(W,S)$ is not triangle free, then
Proposition~\ref{prop:core} and Corollary~\ref{cor:singletonintersection} are
false, as illustrated by the next example.

\begin{example}\label{ex:threecycle2}
Consider the link $\ralpha$ from Example~\ref{ex:threecycle}. We have
$\supp_{\llb 3,5\rrb}(\ralpha) = \{1,3\}$ and $\supp_{\llb 4\rrb}([\ralpha]) =
\{1,2,3\}$, contrary to Proposition~\ref{prop:core}. We also have $\supp_{\llb
6\rrb}([\ralpha]) = \{1,2\}$ so that $\supp_{\llb 4\rrb}([\ralpha]) \cap
\supp_{\llb 6\rrb}([\ralpha]) =\{1,2\}$, which clashes with the conclusion of
Corollary~\ref{cor:singletonintersection}. Once again, this shows that the
assumption that the Coxeter graph has no three-cycles cannot be discarded.
\end{example}

\begin{remark}\label{rem:support for even positions}
Proposition~\ref{prop:core} allows us to assume that $\supp_{\llb
2i\rrb}([\ralpha]) = \{s,t\}$ with $m(s,t)=3$ whenever we have $\llb 2i-1,2i+1
\rrb \in \bs(\ralpha)$ with $\supp_{\llb 2i-1,2i+1\rrb}(\ralpha) = \{s,t\}$.
Moreover, if additionally we have $\llb 2i+1,2i+3 \rrb \in \bs(\ralpha)$, then
we can utilize Corollary~\ref{cor:singletonintersection} to conclude that
$\supp_{\llb 2i+2\rrb}([\ralpha]) = \{t,u\}$ with $m(s,t)=3$. Since $(W,S)$ is
simply laced and triangle free, we know $m(s,u)=2$.  Moreover, we can extend
Proposition~\ref{prop:core} and Corollary~\ref{cor:singletonintersection} to
arbitrary reduced expressions by applying the results to the corresponding link
factors. We will frequently use these facts, and we may do so without
explicitly mentioning the proposition and corollary.
\end{remark}

\end{section}

\begin{section}{Classification of links and braid graphs in Coxeter systems of
type $A_n$}\label{sec:type A}

Our notions of link and link factorization generalize Zollinger's definitions
of string and maximal string decomposition, respectively, for Coxeter systems
of type $A_n$ that appear in~\cite{Zollinger1994a}. Let $(l,k,m,\epsilon)$ be a
quadruple satisfying:
\begin{enumerate}[(a)]
\item $l$ is a positive integer;
\item $k$ is a nonnegative integer less than or equal to $l-1$;
\item $m$ is a positive integer (not necessarily distinct from $l$ or $k$) and;
\item $\epsilon$ is one of $\{+,-, 0\}$;
\end{enumerate}
where $\epsilon=0$ only when $l \leq 2$. From this quadruple, define
$\sigma_{l, k, m, \epsilon}$ with $\epsilon \in \{+,0\}$ to be the word in the
Coxeter system of type $A_n$, called a \emph{string}, as follows. For $l\in
\{1,2\}$, define
\[
\sigma_{1,0,m,0} := s_m, \quad \sigma_{2,0,m,0} := s_m s_{m+1} s_m, \quad
\sigma_{2,1,m,0} = s_{m+1} s_m s_{m+1}.
\]
When $l \geq 3$, then $\sigma_{l,k,m,+}$ is the first $2l-1$ letters from the
following product
\[
s_{m+1} s_m s_{m+2} s_{m+1} s_{m+3} \cdots \UOLunderline{s_{m+k}
s_{m+k-1}}[s_{m+k}] \UOLoverline{s_{m+k+1}s_{m+k}} s_{m+k+2} s_{m+k+1}
s_{m+k+3} s_{m+k+2} \cdots
\]
Note that there are two overlapping opportunities to apply a braid move, each
of which has been underlined or overlined for emphasis. We define the string
$\sigma_{l,k,m,-}$ to be the reverse (i.e., inverse) of $\sigma_{l,l-1-k,m,+}$.
Observe that every string consists of an odd number of letters, namely $2l-1$.
According to~\cite{Zollinger1994a}, each string is a reduced expression.

\begin{example}\label{ex:sigmastringsbraidclass}
Below are six examples of strings in the Coxeter system of type $A_9$:
\begin{center}
\begin{tabular}{ccc}
$\sigma_{6,0,4,+} = 45465768798$ & $\sigma_{6,1,4,+} = 54565768798$ &
$\sigma_{6,2,4,+} = 54656768798$ \\
$\sigma_{6,3,4,+} = 54657678798$ & $\sigma_{6,4,4,+} = 54657687898$ &
$\sigma_{6,5,4,+} = 54657687989$
\end{tabular}
\end{center}
By applying all possible braid moves, one can verify that each expression above
is a link and that these six reduced expressions comprise a single braid class.
\end{example}

The following result is a consequence of Lemmas~1 and 5
in~\cite{Zollinger1994a} and completely characterizes the links in Coxeter
systems of type $A_n$.

\begin{proposition}\label{prop:type A link iff string}
In the Coxeter system of type $A_n$, a reduced expression is a link if and only
if it is a string.
\end{proposition}

In~\cite{Zollinger1994a}, Zollinger refers to the link factorization of a
reduced expression in terms of strings as the \emph{maximal string
decomposition}. Given the structure of strings, we can completely describe the
corresponding braid graphs with ease. The next result is a reformulation of
Lemma~1 from~\cite{Zollinger1994a} in terms of braid graphs.

\begin{proposition}\label{prop:braid graph for string}
For the link $\sigma_{l,k,m,\epsilon}$, we have
\[
\begin{tikzpicture}[every circle node/.style={draw, circle, inner sep=1.25pt}]
\node [] (0) at (-.25,0) {$B(\sigma_{l,k,m,\epsilon})=$};
\node [above right,rotate=45] (label1) at (1,0) {$\sigma_{l,0,m,\epsilon}$};
\node [above right,rotate=45] (label2) at (2,0) {$\sigma_{l,1,m,\epsilon}$};
\node [above right,rotate=45] (label3) at (3,0) {$\sigma_{l,2,m,\epsilon}$};
\node [above right,rotate=45] (label4) at (4,0) {$\sigma_{l,l-2,m,\epsilon}$};
\node [above right,rotate=45] (label5) at (5,0) {$\sigma_{l,l-1,m,\epsilon}$};
\node [circle] (1) at (1,0){};
\node [circle] (2) at (2,0){};
\node [circle] (3) at (3,0){};
\node [circle] (4) at (4,0){};
\node [circle] (5) at (5,0){};
\node at (3.5,0) {$\dots$};
\draw [turq,-, very thick] (1) to (2);
\draw [turq,-, very thick] (2) to (3);
\draw [turq,-, very thick] (4) to (5);
\node [] (period) at (5.25,0) {.};
\end{tikzpicture}
\]
\end{proposition} 

Recall that Corollary~\ref{cor:boxproductsofbraidgraphs} says that the braid
graph for any reduced expression for a group element can be decomposed as the
box product of the braid graphs for the corresponding link factors in the link
factorization. In light of Proposition~\ref{prop:braid graph for string}, we
can be more explicit for Coxeter systems of type $A_n$, as the next result
indicates. This result is also a reformulation of Corollary~6
in~\cite{Zollinger1994a} and can be thought of as a classification of braid
graphs for reduced expressions in Coxeter systems of type $A_n$.

\begin{proposition}\label{prop:braid graphs in type A}
If $\ralpha$ is a reduced expression for a nonidentity $w \in W(A_n)$ with link
factorization $\ralpha_1\mid\ralpha_2\mid\cdots\mid\ralpha_k$ such that each
link $\ralpha_i$ has $2l_i-1$ letters, then
\[
B(\ralpha) \cong
\begin{tabular}{c}
\begin{tikzpicture}[every circle node/.style={draw, circle, inner
sep=1.25pt},decoration={brace,amplitude=7}]
\node [circle] (a) at (2.5,0){};
\node [circle] (b) at (2.5,-1){};
\node (dots) at (2.5,-1.5){$\vdots$};
\node [circle] at (2.5,-2) (c){};
\node [circle] at (2.5,-3)(d){};
\path[-] (a) edge[turq, very thick] (b);
\path[-] (c) edge[turq, very thick] (d);
\draw [decorate] ([xshift=2mm]a.east) --node[right=2mm]{$l_1$}
([xshift=2mm]d.east);

\node at (4,-1.5) {$\Box$};

\node [circle] (e) at (4.5,0){};
\node [circle] (f) at (4.5,-1){};
\node (dots) at (4.5,-1.5){$\vdots$};
\node [circle] at (4.5,-2) (g){};
\node [circle] at (4.5,-3)(h){};
\path[-] (e) edge[turq, very thick] (f);
\path[-] (g) edge[turq, very thick] (h);
\draw [decorate] ([xshift=2mm]e.east) --node[right=2mm]{$l_2$}
([xshift=2mm]h.east);

\node at (6,-1.5) {$\Box$};
\node at (6.50,-1.5) {$\ldots$};
\node at (7,-1.5) {$\Box$};

\node [circle] (i) at (7.5,0){};
\node [circle] (j) at (7.5,-1){};
\node (dots) at (7.5,-1.5){$\vdots$};
\node [circle] at (7.5,-2) (k){};
\node [circle] at (7.5,-3)(l){};
\path[-] (i) edge[turq, very thick] (j);
\path[-] (k) edge[turq, very thick] (l);
\draw [decorate] ([xshift=2mm]i.east) --node[right=2mm]{$l_k$}
([xshift=2mm]l.east);
\end{tikzpicture}
\end{tabular},
\]
where the $i$th link factor in the decomposition is a path graph with $l_i$
vertices.
\end{proposition}

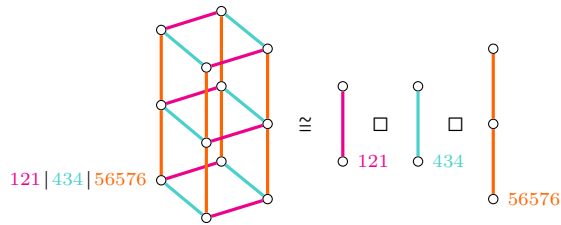
\begin{figure}[ht!]
\begin{tikzpicture}[every circle node/.style={draw, circle, fill=none ,inner
sep=1.25pt}]
\node [circle] (1) at (8.414,1.75){};
\node [circle] (2) [label=right: $\scriptstyle \textcolor{magenta}{121}$] at
(8.414,.75){};
\path[-] (1) edge[magenta, very thick] (2);

\node (x1) at (7.914,1.25) {$\cong$};

\node [circle] (3) at (9.414,1.75){};
\node [circle] (4) [label= right: $\scriptstyle \textcolor{turq}{434}$] at
(9.414,.75){};
\path[-] (3) edge[turq, very thick] (4);

\node (x2) at (8.914,1.25) {$\Box$};

\node [circle] (5) at (10.414,2.25){};
\node [circle] (6) at (10.414,1.25){};
\node [circle] (7) [label= right: $\scriptstyle \textcolor{orange2}{56576}$] at
(10.414,.25){};
\path[-] (5) edge[orange2, very thick]  (6);
\path[-] (6) edge[orange2, very thick]  (7);

\node (e) at (9.914,1.25) {$\Box$};

\node [circle] (8) at (6,2.5){};
\node [circle] (9) at (6,1.5){};
\node [circle] (10) [label=left: $\scriptstyle \textcolor{magenta}{121} \, \mid
\, \textcolor{turq}{434} \, \mid \, \textcolor{orange2}{56576} $]  at (6,.5){};
\path[-] (8) edge[orange2, very thick]  (9);
\path[-] (9) edge[orange2, very thick]  (10);

\node [circle] (11) at (6.6,2){};
\node [circle] (12) at (6.6,1){};
\node [circle] (13) at (6.6,0){};

\node [circle] (14) at (7.414,2.25){};
\node [circle] (15) at (7.414,1.25){};
\node [circle] (16) at (7.414,.25){};
\path[-] (14) edge[orange2, very thick] (15);

\path[-] (10) edge[turq, very thick]  (13); 
\path[-] (13) edge[magenta, very thick]  (16); 

\node [circle] (20) at (6.8,2.75){}; 
\node [circle] (21) at (6.8,1.75){}; 
\node [circle] (22) at (6.8,.75){}; 
\path[-] (8) edge[magenta, very thick]  (20); 
\path[-] (20) edge[turq, very thick]  (14);

\path[draw] (20) edge[orange2, very thick] (21); 
\path[draw] (21) edge[orange2, very thick]  (22);

\path[draw] (9) edge[magenta, very thick]  (21);
\path[draw] (10) edge[magenta, very thick]  (22);

\path[draw] (21) edge[turq, very thick]  (15);
\path[draw] (22) edge[turq, very thick]  (16);
\path[-] (11) edge[magenta, very thick]  (14);
\path[-] (12) edge[magenta, very thick]  (15);
\path[-] (11) edge[orange2, very thick] (12);

\path[-] (8) edge[turq, very thick]  (11);

\path[-] (9) edge[turq, very thick]  (12);
\path[-] (15) edge[orange2, very thick]  (16);
\path[-] (12) edge[orange2, very thick]  (13);
\end{tikzpicture}\caption{Braid graph for the reduced expression from
Example~\ref{ex:product of paths} and its decomposition into a box product of
path graphs.}\label{fig:product of paths}
\end{figure}

\begin{example}\label{ex:product of paths}
The word $\ralpha = 12143456576$ is a reduced expression for some element in
$W(A_7)$. The link factorization for $\ralpha$ is $\textcolor{magenta}{121}
\mid\textcolor{turq}{434} \mid \textcolor{orange2}{56576}$. The resulting braid
graph and its decomposition into a box product of path graphs is shown in
Figure~\ref{fig:product of paths}.
\end{example}

\end{section}

\begin{section}{Braid graphs as partial cubes}\label{sec:embedding}

If $G$ is a graph, let $V(G)$ denote the vertex set of $G$. If $S\subseteq
V(G)$ is any subset of vertices of $G$, then we define the \emph{induced
subgraph} $G[S]$ to be the graph whose vertex set is $S$ and whose edge set
consists of all of the edges of $G$ that have endpoints in $S$. An
\emph{embedding} of a simple graph $G$ into a simple graph $H$ is an injection
$f:V(G)\to V(H)$ with the property that if $u$ and $v$ are adjacent vertices in
$G$, then $f(u)$ and $f(v)$ are adjacent in $H$. If in addition, $f(u)$ and
$f(v)$ adjacent in $H$ implies $u$ and $v$ adjacent in $G$, then we say that
$f$ is an \emph{induced embedding}. In the graph theory literature, an
embedding is often referred to as a monomorphism while an induced embedding is
a faithful monomorphism~\cite{Hahn1997}. If $f$ is an induced embedding, then
$G$ is isomorphic to the subgraph of $H$ induced by the image of $f$. That is,
$G\cong H[\im(f)]$.

We can view any connected graph $G$ as a metric space by taking the standard
geodesic metric. That is, the distance between $u,v\in V(G)$ is defined via
\[
d_G(u,v):=\text{length of any minimal path between }u\text{ and }v.
\]
An \emph{isometric embedding} of $G$ into $H$ is a function $f:V(G)\to V(H)$
with the property that $d_G(u,v) = d_H(f(u),f(v))$ for all $u,v \in V(G)$.
Since an isometry is injective and two vertices are adjacent if and only if the
distance between them is one, every isometric embedding is also an induced
embedding. It is worth mentioning that an induced embedding is not necessarily
an isometric embedding. As an example, consider an embedding of a path with
four vertices into a cycle with five vertices. This embedding is an induced
embedding but is not an isometric embedding.

For a nonnegative integer $r$ we will denote the set of binary strings of
length $r$ by $\{0,1\}^r$. That is,
\[
\{0,1\}^r := \{a_1a_2\cdots a_r \mid a_k \in \{0,1\} \}.
\]
The \emph{hypercube graph} of dimension $r\geq 0$, denoted by $Q_r$, is defined
to be the graph whose vertices are elements of $\{0,1\}^r$ with two binary
strings connected by an edge exactly when they differ by a single digit (i.e.,
the Hamming distance between the two vertices is equal to one). Note that $Q_0$
consists of a single vertex labeled by the empty string. A \emph{partial cube}
is a graph that can be isometrically embedded into a hypercube. The
\emph{isometric dimension} of a partial cube is the minimum dimension of a
hypercube into which it may be isometrically embedded. That is, the isometric
dimension of a partial cube $G$ is the nonnegative integer
\[
\dim_I(G):=\min\{m \in \N\cup\{0\} \mid \text{there exists an isometric
embedding of $G$ into $Q_m$}\}.
\]
The following proposition is a result from~\cite{ovchinnikov2008partial}.

\begin{proposition}\label{prop:boxprodpartial}
If $G_1$ and $G_2$ are partial cubes, then $G_1\square G_2$ is a partial cube.
Moreover, $\dim_I(G_1\square G_2)=\dim_I(G_1)+\dim_I(G_2)$.
\end{proposition}

Assume $(W,S)$ is a simply-laced triangle-free Coxeter system and suppose
$\ralpha$ is a reduced expression for some $w\in W$. The goal of this section
is to establish an isometric embedding of $B(\ralpha)$ into
$Q_{\rank(\ralpha)}$ (see Theorem~\ref{thm:braid graphs are partial cubes}).
The crux is proving that every link can be isometrically embedded in a
hypercube whose dimension is at most the rank of the link (see
Proposition~\ref{prop:braid graph for a link is a partial cube}). Then we can
simply apply the decomposition for reduced expressions given in
Corollary~\ref{cor:boxproductsofbraidgraphs} together with
Proposition~\ref{prop:boxprodpartial} to obtain the result for arbitrary
reduced expressions. It also follows that the isometric dimension of every
braid graph is bounded above by the number of braid shadows.

Throughout this section, we will utilize many of the properties of braid
equivalent reduced expressions and links that were developed in
Section~\ref{sec:architecture}. We begin with several technical lemmas that
will be useful in the the proof of Proposition~\ref{prop:braid graph for a link
is a partial cube}. Our first lemma indicates that the left and right ends of a
link are fairly rigid in structure.

\begin{lemma}\label{lem:firstlastletter}
Suppose $(W, S)$ is a simply-laced triangle-free Coxeter system and let
$\ralpha$ be a link of rank $r\geq 1$.
\begin{enumerate}[(a)]
\item If $\supp_{\llb 2\rrb}([\ralpha]) = \{s,t\}$ with $m(s,t)=3$, then
$\ralpha_{\llb 1,2\rrb}=st$ or $\ralpha_{\llb 1,2\rrb}=ts$.
\item If $\supp_{\llb 2r\rrb}([\ralpha]) = \{s,t\}$ with $m(s,t)=3$, then
$\ralpha_{\llb 2r,2r+1\rrb}=st$ or $\ralpha_{\llb 2r,2r+1\rrb}=ts$.
\end{enumerate}
\end{lemma}

\begin{proof}
Suppose $\ralpha$ is a link such that $\supp_{\llb 2\rrb}([\ralpha]) = \{s,t\}$
with $m(s,t)=3$ and let $\ralpha_{\llb 1,2\rrb}=us$. Note that $u \neq s$ since
$\ralpha$ is reduced. Now, consider the set
\[
X = \big\{\rbeta \in [\ralpha]\mid \llb 1,3\rrb \in \bs(\rbeta) \text{ and }
u\in \supp_{\llb 1,3\rrb}(\rbeta)\big\}.
\]
Since $\ralpha$ is a link, $X$ is nonempty. Choose any $\rbeta \in X$. Then
$\supp_{\llb 1,3\rrb}(\rbeta) = \{s,t\}$ by Proposition~\ref{prop:core}, and so
$u = t$. This proves Part~(a). Part~(b) follows from a symmetric argument.
\end{proof}

The next lemma states that the support of the common position of two
overlapping braid shadows has cardinality three and places restrictions on the
local structure of overlapping braid shadows.

\begin{lemma}\label{lem:TeamJJJnew}
Suppose $(W, S)$ is a simply-laced triangle-free Coxeter system. If $\ralpha$
is a link of rank $r\geq 2$ such that for $1\leq i\leq r-1$, $\supp_{\llb
2i\rrb}([\ralpha]) = \{s,t\}$ and $\supp_{\llb 2i+2\rrb}([\ralpha]) = \{t,u\}$
with $m(s,t)=3=m(t,u)$ according to Remark~\ref{rem:support for even
positions}, then $\supp_{\llb 2i+1\rrb}([\ralpha]) =\{s,t,u\}$, $\ralpha_{\llb
2i\rrb}\neq \ralpha_{\llb 2i+2\rrb}$, and $\ralpha_{\llb 2i+1\rrb}\in
\{s,t,u\}\setminus \{\ralpha_{\llb 2i\rrb}, \ralpha_{\llb 2i+2\rrb}\}$.
\end{lemma}

\begin{proof}
Let $\ralpha$ be a link of rank $r\geq 2$ such that $\supp_{\llb
2i\rrb}([\ralpha]) = \{s,t\}$ and $\supp_{\llb 2i+2\rrb}([\ralpha]) = \{t,u\}$
with $m(s,t)=3=m(t,u)$. It is clear that $\{s,t,u\} \subseteq \supp_{\llb
2i+1\rrb}([\ralpha])$. Now, let $v \in \supp_{\llb 2i+1\rrb}([\ralpha])$ and
consider the set
\[
X = \big\{\rbeta\in[\ralpha]\mid\supp_{\llb 2i+1\rrb}(\rbeta) = \{v\}\big\}.
\]
Since $\supp_{\llb 2i+1\rrb}([\ralpha])\neq \{v\}$, there exists $\rbeta \in X$
such that either $\llb 2i-1,2i+1\rrb \in \bs(\rbeta)$ or $\llb 2i+1,2i+3\rrb
\in \bs(\rbeta)$. Then $v \in \{s,t\}$ or $v \in \{t,u\}$ according to
Proposition~\ref{prop:equalsupport}. This proves the first claim. If
$\ralpha_{\llb 2i\rrb}= \ralpha_{\llb 2i+2\rrb}$, then it must be the case that
$t$ occupies positions $2i$ and $2i+2$ in $\ralpha$.  But this would imply that
$\ralpha_{\llb 2i+1\rrb}\in\{s,u\}$, which violates Proposition~\ref{prop:hugh}
in either case. Thus, $\ralpha_{\llb 2i\rrb}\neq \ralpha_{\llb 2i+2\rrb}$. It
follows that $\ralpha_{\llb 2i+1\rrb}\in \{s,t,u\}\setminus \{\ralpha_{\llb
2i\rrb}, \ralpha_{\llb 2i+2\rrb}\}$, otherwise $\ralpha$ would not be reduced.
\end{proof}

One consequence of Lemma~\ref{lem:TeamJJJnew} is that for any two overlapping
braid shadows in a braid chain $[\ralpha]$, there are three possible forms that
$\ralpha_{\llb 2i,2i+2\rrb}$ may take:
\begin{enumerate}[(a)]
\item
$\displaystyle\underbrace{\cdots\lfrac{?}{2i-1}\lfrac{s}{2i}\lfrac{u}{2i+1
}\lfrac{t}{2i+2}\lfrac{?}{2i+3}\cdots}_{\ralpha}$
\item
$\displaystyle\underbrace{\cdots\lfrac{?}{2i-1}\lfrac{s}{2i}\lfrac{t}{2i+1
}\lfrac{u}{2i+2}\lfrac{?}{2i+3}\cdots}_{\ralpha}$
\item
$\displaystyle\underbrace{\cdots\lfrac{?}{2i-1}\lfrac{t}{2i}\lfrac{s}{2i+1
}\lfrac{u}{2i+2}\lfrac{?}{2i+3}\cdots}_{\ralpha}$
\end{enumerate}
where $m(s,t)=3=m(t,u)$ and $m(s,u)=2$. Note that position $2i+1$ is the
location where the two braid shadows in $[\ralpha]$ overlap.

The next lemma states that for every link $\ralpha$ and every pair of
overlapping braid shadows for $[\ralpha]$, there exists a link in $[\ralpha]$
such that overlapping braid shadows occur simultaneously.

\begin{lemma}\label{lem:structure of overlapping braids}
If $(W, S)$ is a simply-laced triangle-free Coxeter system and $\ralpha$ is a
link of rank $r\geq 2$, then for all $1\leq i\leq r-1$ there exists $\rsigma
\in [\ralpha]$ with the property that $\llb 2i-1,2i+1\rrb,\llb 2i+1,2i+3\rrb
\in \bs(\rsigma)$.
\end{lemma}

\begin{proof}
This result follows from Lemma~\ref{lem:TeamJJJnew}.
\end{proof}

Equivalently, the previous lemma states that if $(W,S)$ is a simply-laced
triangle-free Coxeter system and $\ralpha$ is a link of rank at least two such
that $\supp_{\llb 2i\rrb}([\ralpha]) = \{s,t\}$ and $\supp_{\llb
2i+2\rrb}([\ralpha]) = \{t,u\}$ with $m(s,t)=3=m(t,u)$, then there exists a
link $\rsigma \in [\ralpha]$ with the property that $\rsigma_{\llb
2i-1,2i+3\rrb} = tstut$.

Given a reduced expression $\ralpha$ consisting of at least two letters, let
$\hatralpha$ be the reduced expression obtained from $\ralpha$ by deleting the
rightmost two letters. If $\ralpha$ is a link, we should not expect
$\hatralpha$ to be a link. However, the next lemma indicates that we will
obtain a link if we delete the rightmost two letters from a carefully chosen
link. Certainly, we also have a ``left-handed" version of this lemma. There is
likely a generalized version of this lemma where we delete all the letters on
either side of the position where two braid shadows overlap, however, this
stronger version is not needed for our purposes.

\begin{lemma}\label{lem:delete last two letters from special link}
Suppose $(W, S)$ is a simply-laced triangle-free Coxeter system and $\ralpha$
is a link of rank $r \geq 2$ and choose $\rsigma\in[\ralpha]$ such that $\llb
2r-3,2r-1\rrb,\llb 2r-1,2r+1\rrb \in \bs(\rsigma)$ according to
Lemma~\ref{lem:structure of overlapping braids}. Then $\hatrsigma$ is a link of
rank $r-1$, and if $\rbeta\in [\ralpha]$ such that $\supp_{\llb
2r\rrb}(\rbeta)=\supp_{\llb 2r\rrb}(\rsigma)$, then $\hatrbeta\in
[\hatrsigma]$. Moreover, every element of $[\hatrsigma]$ is of the form
$\hatrbeta$ for some $\rbeta \in [\ralpha]$ satisfying $\supp_{\llb
2r\rrb}(\rbeta)=\supp_{\llb 2r\rrb}(\rsigma)$.
\end{lemma}

\begin{proof}
The fact that $\hatrsigma$ is a link of rank $r-1$ is easily seen.  Let
$\rbeta\in [\ralpha]$ such that $\supp_{\llb 2r\rrb}(\rbeta)=\supp_{\llb
2r\rrb}(\rsigma)$.  It follows from Lemma~\ref{lem:firstlastletter} that
$\rbeta_{\llb 2r,2r+1\rrb}=\rsigma_{\llb 2r,2r+1\rrb}$.  Any minimal sequence
of braid moves that transforms $\rsigma$ into $\rbeta$ does not involve the
generators appearing in the rightmost braid shadow, and hence the same sequence
of braid moves will transform $\hatrsigma$ into $\hatrbeta$.  Thus,
$\hatrbeta\in [\hatrsigma]$. On the other hand, suppose that $\rsigma_{\llb
2r,2r+1\rrb} = ut$ and let $\rgamma \in [\hatrsigma]$. Then there is a sequence
of braid moves that transforms $\hatrsigma$ into $\rgamma$. Set $\rbeta :=
\rgamma ut$ so that $\rgamma = \hatrbeta$. When applied to $\rsigma$ instead of
$\hatrsigma$, this sequence of braid moves transforms $\rsigma$ into $\rbeta$.
It follows from Proposition~\ref{prop:matsumoto} that $\rbeta$ is reduced and
hence $\rbeta \in [\ralpha]$. Clearly, $\supp_{\llb 2r\rrb}(\rbeta) = \{u\} =
\supp_{\llb 2r\rrb}(\rsigma)$, which completes the proof.
\end{proof} 

Our next lemma has the important conclusion that every link is uniquely
determined by the generators appearing at the even indices of the word.
Moreover, every factor that begins and ends at an even index is uniquely
determined by the even-index generators that appear in the factor.
 
\begin{lemma}\label{lem:equalexp}
Suppose $(W, S)$ is a simply-laced triangle-free Coxeter system and let
$\ralpha$ and $\rbeta$ be two braid equivalent links of rank $r\geq 1$. Then
$\ralpha = \rbeta$ if and only if $\ralpha_{\llb 2i\rrb}=\rbeta_{\llb 2i\rrb}$
for all $1\leq i\leq r$.
\end{lemma}

\begin{proof}
The forward implication is immediate. Conversely, assume that $\ralpha_{\llb
2i\rrb}=\rbeta_{\llb 2i\rrb}$ for all $1\leq i\leq r$. The fact that
$\ralpha_{\llb 2i+1\rrb}=\rbeta_{\llb 2i+1\rrb}$ for all $1\leq i\leq r-1$
follows from Lemma~\ref{lem:TeamJJJnew}. Thus, we have $\ralpha_{\llb
2,2r\rrb}=\rbeta_{\llb 2,2r\rrb}$. Applying Lemma~\ref{lem:firstlastletter}
then yields the desired conclusion.
\end{proof}

The following notation will be useful in the next lemma. If $\ralpha$ is a link
such that $\llb 2i-1,2i+1 \rrb \in \bs(\ralpha)$, then let $b_{\llb
2i-1,2i+1\rrb}(\ralpha)$ denote the link obtained from $\ralpha$ by applying
the braid move in positions $\llb 2i-1,2i+1 \rrb$.  We certainly have $b_{\llb
2i-1,2i+1\rrb}(\ralpha) \in [\ralpha]$ with $d(b_{\llb
2i-1,2i+1\rrb}(\ralpha),\ralpha) = 1$. Suppose $\ralpha$ is a link of rank
$r\geq 1$. Given any $\rgamma \in [\ralpha]$, we associate subsets $X_{\rgamma}
,Y_{\rgamma}\subseteq [\ralpha]$ defined as follows:
\[
\begin{array}{ccc}
X_{\rgamma} : = \{\rbeta \in [\ralpha] \mid \supp_{\llb 2r\rrb}(\rbeta) =
\supp_{\llb 2r\rrb}(\rgamma) \}  & \text{and} & Y_{\rgamma} : = \{\rbeta \in
[\ralpha] \mid \supp_{\llb 2r\rrb}(\rbeta) \neq \supp_{\llb 2r\rrb}(\rgamma)\}.
\end{array}
\]
The set $X_\rgamma$ is simply the collection of links in $[\ralpha]$ that share
the same letter as $\rgamma$ in the second position from the right while
$Y_\rgamma$ is the complement of $X_\rgamma$ relative to $[\ralpha]$.  In
particular, if $\supp_{\llb 2r\rrb}([\ralpha])=\{s,t\}$ and $\supp_{\llb
2r\rrb}(\rgamma)=\{s\}$, then every link in $X_{\rgamma}$ has $s$ in position
$2r$ while every link in $Y_{\rgamma}$ has $t$ in position $2r$. It follows
from Lemma~\ref{lem:firstlastletter} that if $\rbeta\in X_\rgamma$, then
$\rbeta_{\llb 2r,2r+1\rrb}=\rgamma_{\llb 2r,2r+1\rrb}$. That is, every pair of
links in $X_\rgamma$ agree on the last two letters.  Similarly, every pair of
links in $Y_\rgamma$ agree on the last two letters.

\begin{example}
Consider the links presented in Parts~(a) and~(c) of Example~\ref{ex:braid
graphs}. We see that $X_{\ralpha_3}=\{\ralpha_1, \ralpha_2, \ralpha_3\}$ and
$Y_{\ralpha_3}=\{\ralpha_4\}$ while $X_{\rgamma_4}=\{\rgamma_3,\rgamma_4,
\rgamma_5\}$ and $Y_{\rgamma_4}=\{\rgamma_1, \rgamma_2\}$.
\end{example}

\begin{lemma}\label{lem:partition of braid class for a link}
Suppose $(W,S)$ is a simply-laced triangle-free Coxeter system and $\ralpha$ is
a link of rank $r\geq 2$. Choose $\rsigma\in [\ralpha]$ such that $\llb
2r-3,2r-1\rrb,\llb 2r-1,2r+1\rrb \in \bs(\rsigma)$ according to
Lemma~\ref{lem:structure of overlapping braids} and let $X_\rsigma$ and
$Y_\rsigma$ be defined as above. Then
\begin{enumerate}[(a)]
\item $\{X_\rsigma,Y_\rsigma\}$ is a partition of $[\ralpha]$. 
\item If $\rbeta \in X_\rsigma$, then $\hatrbeta \in [\hatrsigma]$. Moreover,
every element of $[\hatrsigma]$ is of the form $\hatrbeta$ for some $\rbeta \in
X_\rsigma$.
\item If $\rbeta \in Y_\rsigma$, then $\llb 2r-1,2r+1\rrb \in \bs(\rbeta)$ and
$\left(b_{\llb 2r-1,2r+1\rrb}(\rbeta)\right)_{\llb 1,2r-1\rrb} \in
[\hatrsigma]$.
\end{enumerate}
\end{lemma}

\begin{proof}
Certainly, $X_\sigma$ and $Y_\sigma$ are disjoint and $X_\rsigma\cup
Y_\rsigma=[\ralpha]$. Moreover, $X_{\rsigma}$ is nonempty since $\rsigma\in
X_{\rsigma}$.  Applying the braid move occurring in the rightmost braid shadow
in $\rsigma$ results in a link in $Y_{\rsigma}$, and so $Y_{\rsigma}$ is also
nonempty.  Thus, $\{X_\rsigma,Y_\rsigma\}$ is a partition of $[\ralpha]$, which
verifies Part~(a). Part~(b) follows immediately from Lemma~\ref{lem:delete last
two letters from special link}.

Using Lemma~\ref{lem:structure of overlapping braids}, we can write
$\rsigma_{\llb 2r-3,2r+1\rrb} = tstut$, where $\supp_{\llb
2r-2\rrb}([\ralpha])=\{s,t\}$ and $\supp_{\llb 2r\rrb}([\ralpha])=\{t,u\}$ with
$m(s,t)=3=m(t,u)$. Let $\rbeta\in Y_{\rsigma}$, so that $\rbeta_{\llb
2r,2r+1\rrb}=tu$ by Lemma~\ref{lem:firstlastletter}. By applying
Lemma~\ref{lem:TeamJJJnew} with $i=r-1$, we have $\supp_{\llb
2r-1\rrb}([\ralpha]) = \{s,t,u\}$.  By considering all possibilities for
$\rbeta_{\llb 2r-2,2r+1\rrb}$ and using the fact that $\rbeta$ is reduced while
$\llb 2r-2,2r\rrb \notin \bs([\ralpha])$, we can conclude that $\rbeta_{\llb
2r-2,2r+1\rrb} = sutu$. Since $m(t,u) = 3$, it follows that $\llb 2r-1,2r+1\rrb
\in \bs(\rbeta)$.  But then $b_{\llb 2r-1,2r+1\rrb}(\rbeta) \in X_\rsigma$, and
hence $\left(b_{\llb 2r-1,2r+1\rrb}(\rbeta)\right)_{\llb 1,2r-1\rrb} \in
[\hatrsigma]$ by Part~(b). This verifies Part~(c).
\end{proof}

Suppose $\ralpha$ is a link of rank $r\geq 2$ and choose $\rsigma\in [\ralpha]$
such that $\llb 2r-3,2r-1\rrb,\llb 2r-1,2r+1\rrb \in \bs(\rsigma)$ according to
Lemma~\ref{lem:structure of overlapping braids}.  In light of Part~(c) of
Lemma~\ref{lem:partition of braid class for a link}, for $\rbeta\in
Y_{\rsigma}$, define $\checkrbeta:=\left(b_{\llb
2r-1,2r+1\rrb}(\rbeta)\right)_{\llb 1,2r-1\rrb}$.  That is, for $\rbeta\in
Y_{\rsigma}$, $\checkrbeta$ is the link we obtain in $[\hatrsigma]$ by first
applying the braid move available in the rightmost braid shadow of $\rbeta$ and
then deleting the two rightmost letters of the resulting expression.

An important corollary of the preceding lemma is that the braid graph for a
link $\ralpha$ is obtained by gluing together two induced subgraphs, one of
which is the image of an isometric embedding of $B(\hatrsigma)$ into
$B(\ralpha)$. In order to state this precisely, we need the following
definition.

\begin{definition}\label{def:braidgraph decomposition embeddings for a link}
Suppose $(W,S)$ is a simply-laced triangle-free Coxeter system and $\ralpha$ is
a link of rank $r\geq 2$. Choose $\rsigma\in [\ralpha]$ such that $\llb
2r-3,2r-1\rrb,\llb 2r-1,2r+1\rrb \in \bs(\rsigma)$ according to
Lemma~\ref{lem:structure of overlapping braids}, and suppose that
$\rsigma_{\llb 2r-3,2r+1\rrb} = tstut$, where $m(s,t)=3=m(t,u)$. Let
$\Omega_{\ralpha}^{\rsigma}:[\hatrsigma] \to [\ralpha]$ be the function that
conjoins the letters $ut$ on the right of each element of $[\hatrsigma]$.
\end{definition}

Note that the function $\Omega_{\ralpha}^{\rsigma}$ is well defined by
Lemma~\ref{lem:partition of braid class for a link}.

\begin{corollary}\label{cor:subgraph of a link}
Suppose $(W,S)$ is a simply-laced triangle-free Coxeter system and $\ralpha$ is
a link of rank $r\geq 2$ and assume the notation of
Definition~\ref{def:braidgraph decomposition embeddings for a link}. Then
\begin{enumerate}[(a)]
\item\label{cor:subgraph of a link a} The function
$\Omega_{\ralpha}^{\rsigma}:[\hatrsigma] \to [\ralpha]$ is an isometric
embedding from $B(\hatrsigma)$ to $B(\ralpha)$. In particular,
$\im(\Omega_{\ralpha}^{\rsigma}) = X_\rsigma$ and $B(\hatrsigma)$ is isomorphic
to the induced subgraph $B(\ralpha)[X_\rsigma]$.
\item\label{cor:subgraph of a link b} The induced subgraph
$B(\ralpha)[Y_\rsigma]$ is an isometric subgraph of $B(\ralpha)$.
\item\label{cor:subgraph of a link c} If $\rbeta \in X_\rsigma$ and  $\rgamma
\in Y_\rsigma$, then $d_{B(\ralpha)}(\rbeta ,\rgamma) = d_{B(\ralpha)}(\rbeta
,b_{\llb 2r-1,2r+1\rrb}(\rgamma)) + 1$.
\end{enumerate}  
\end{corollary}

\begin{proof}
Using Lemma~\ref{lem:partition of braid class for a link}, choose
$\hatrbeta,\hatrgamma \in [\hatrsigma]$ for $\rbeta,\rgamma \in [\ralpha]$.
Then clearly
\[ 
d_{B(\hatrsigma)}(\hatrbeta, \hatrgamma) = d_{B(\ralpha)}(\hatrbeta ut,
\hatrgamma ut) = d_{B(\ralpha)}(\Omega_\ralpha^\rsigma(\hatrbeta),
\Omega_\ralpha^\rsigma(\hatrgamma)).
\]  
The statement $\im(\Omega_{\ralpha}^{\rsigma}) = X_\rsigma$ follows immediately
from Lemma~\ref{lem:partition of braid class for a link}(b). The image of an
isometric embedding is the subgraph induced by the image, and so $B(\hatrsigma)
\cong B(\ralpha)[X_\rsigma]$. This verifies Part~\ref{cor:subgraph of a link
a}.

Now, let $\rbeta,\rgamma \in Y_\rsigma$. Any minimal sequence of braid moves
that transforms $\rbeta$ into $\rgamma$ does not involve the braid shadow $\llb
2r-1,2r+1\rrb$. Therefore, every shortest path between $\rbeta$ and $\rgamma$
is contained in $B(\ralpha)[Y_\rsigma]$. Thus, $B(\ralpha)[Y_\sigma]$ is convex
and hence an isometric subgraph of $B(\ralpha)$. This proves
Part~\ref{cor:subgraph of a link b}.

Finally, let $\rbeta \in X_\sigma$ and $\rgamma \in Y_\rsigma$ and choose a
minimal sequence of braid moves $b_1,b_2,\ldots,b_k$ that transforms $\rbeta$
into $b_{\llb 2r-1,2r+1\rrb}(\rgamma)$. Note that the latter element is a
member of $X_{\rsigma}$ by Lemma~\ref{lem:partition of braid class for a
link}(c). Then $b_1,b_2,\ldots,b_k, b_{\llb 2r-1,2r+1\rrb}$ is a minimal
sequence of braid moves that transforms $\rbeta$ into $\rgamma$, and hence
\[
d_{B(\ralpha)}(\rbeta ,\rgamma) = k + 1 = d_{B(\ralpha)}(\rbeta ,b_{\llb
2r-1,2r+1\rrb}(\rgamma)) + 1.
\]
This completes the proof of Part~\ref{cor:subgraph of a link c}.
\end{proof}

Corollary~\ref{cor:subgraph of a link} explicitly states that the induced
subgraph $B(\ralpha)[X_\rsigma]$ is the braid graph for some link.  We
conjecture that the induced subgraph $B(\ralpha)[Y_\rsigma]$ is also the braid
graph for some link. We prove this conjecture for Fibonacci links in
Section~\ref{sec:fibonacci}.

\begin{example}\label{ex:braid graph with partition}
Consider the link $\ralpha = 32313435464$ in the simply-laced triangle-free
Coxeter system of type $\widetilde D_5$, whose Coxeter graph is shown in
Figure~\ref{fig:labeledgraphs}. One possible choice for a link satisfying
Definition~\ref{def:braidgraph decomposition embeddings for a link} is
$\rsigma= 32314345464$. The braid graph for $\ralpha$ is given in
Figure~\ref{fig:BoxWithTwoFlaps}. We have highlighted $B(\ralpha)[X_{\rsigma}]$
and $B(\ralpha)[Y_{\rsigma}]$ in \textcolor{green}{green} and
\textcolor{magenta}{magenta}, respectively.  In this case,
$\hatrsigma=323143454$, and in agreement with Corollary~\ref{cor:subgraph of a
link}, we have $B(\hatrsigma)\cong B(\ralpha)[X_{\rsigma}]$. We encountered the
braid graph for $B(\hatrsigma)$ in Example~\ref{ex:braid graphs}(d). Moreover,
we see that $B(\ralpha)[Y_{\rsigma}]$ happens to be isomorphic to the braid
graph for $3231343$, which is the link we obtain by deleting the last four
letters from $b_{\llb 9,11 \rrb}(\ralpha)=32313435646$. This is the braid graph
for $\rgamma_4$ that we saw in Example~\ref{ex:braid graphs}(c). Each of the
edges joining $B(\ralpha)[X_{\rsigma}]$ and $B(\ralpha)[Y_{\rsigma}]$
correspond to the braid move applied in the rightmost braid shadow.
\end{example}

\begin{figure}[ht!]
			\centering
\begin{tikzpicture}[every circle node/.style={draw, circle, inner sep=1.25pt}]
\node [circle] (1) [label=left:$\scriptstyle b_{\llb 9,11\rrb}(\ralpha){\ =\
}32313435646$\ \ ] at (0,1){};
				\node [circle] (2) at (.707,.293){};
				\node [circle] (3) at (-.707,1.707){};
\node [circle] (4) [label=right:$\; \scriptstyle 32313435464{\ =\ }{\ralpha}$]
at (.707,1.707){};
				\node [circle] (5) at (0,2.414){};
				\draw [rotate = 90, magenta,-, very thick] (1) to (2);
				\draw [magenta,-, very thick] (1) to (3);
				\draw [turq,-, very thick] (1) to (4);
				\draw [turq,-, very thick] (3) to (5);
				\draw [green,-, very thick] (5) to (4);
				\node [circle] (9) at (1.414,1){};
				
				\node [circle] (10) at (0,1.707){};
				\node [circle] (11) at (0,3.121){};
\node [circle] (13) [label=right:$\; \scriptstyle 32314345464{\ = \ }\rsigma$]
at (.707,2.414){};
				\node [circle] (14) at (-.707,2.414){};
				\node [circle] (15) at (1.414,3.121){};
				\node [circle] (16) at (.707,3.828){};
				\draw [magenta,-, very thick] (1) to (10);
				\draw [turq,-, very thick] (10) to (13);
				\draw [magenta,-, very thick] (10) to (14);
				\draw [turq,-, very thick] (14) to (11);
				\draw [green,-, very thick] (13) to (11);
				\draw [green,-, very thick] (13) to (4);
				\draw [green,-, very thick] (4) to (9);
				\draw [turq,-, very thick] (9) to (2);
				\draw [green,-, very thick] (5) to (11);
				\draw [magenta,-, very thick] (14) to (3);
				\draw [green,-, very thick] (15) to (16);
				\draw [green,-, very thick] (11) to (16);
				\draw [green,-, very thick] (15) to (13);
\begin{pgfonlayer}{background}
\highlight{8pt}{green}{(9.center) to (4.center) to (5.center) to (11.center) to
(16.center) to (15.center) to (13.center) to (11.center)}
\highlight{8pt}{green}{(13.center) to (4.center)}
\highlight{8pt}{magenta}{(2.center) to (1.center) to (3.center) to (14.center)
to (10.center) to (1.center)}
\end{pgfonlayer}
			\end{tikzpicture}
\caption{Braid graph for the reduced expression in Example~\ref{ex:braid graph
with partition} together with a partition of the vertices according to
Lemma~\ref{lem:partition of braid class for a link}.}
	\label{fig:BoxWithTwoFlaps}
\end{figure}
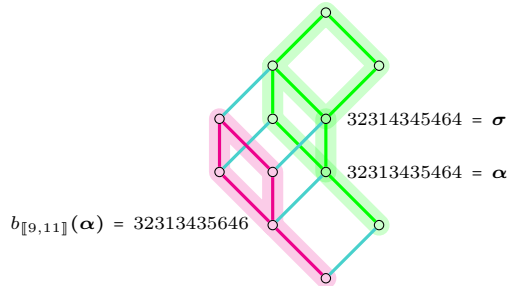

Recall that the crux of this section is to prove that every link can be
isometrically embedded in a hypercube whose dimension is at most the rank of
the link. We utilize the following definition for this purpose.

\begin{definition}\label{def:embedding}
Suppose $(W,S)$ is a simply-laced triangle-free Coxeter system and let
$\ralpha$ be a link of rank $r\geq 0$. For $r\geq 1$, define
$\Phi_{\ralpha}:[\ralpha] \to \{0,1\}^r$ via
$\Phi_{\ralpha}(\rbeta)=a_1a_2\cdots a_r$, where
\[
a_k = \begin{cases} 0, & \text{if} \  \supp_{\llb 2k\rrb}(\rbeta) = \supp_{\llb
2k\rrb}(\ralpha)\\
1, & \text{if} \ \supp_{\llb 2k\rrb}(\rbeta) \neq \supp_{\llb 2k\rrb}(\ralpha).
\end{cases}
\]
In the case that $r = 0$, the unique element (consisting of a single letter) in
the braid class is mapped to the empty binary string.
\end{definition}

It is worth pointing out that the definition of the map $\Phi_{\ralpha}$
depends on the choice of link. Choosing a different representative of
$[\ralpha]$ will necessarily result in a different mapping. However, any two
such embeddings differ only by an automorphism of the hypercube, cf.
Lemma~\ref{lem:isometric embedding lemma}.

\begin{lemma}\label{lem:isometric embedding lemma} 
Suppose $(W,S)$ is a simply-laced triangle-free Coxeter system. If $\ralpha$
and $\rbeta$ are braid equivalent links of rank $r\geq 0$, then there is an
automorphism $F:\{0,1\}^r\to \{0,1\}^r$ of the hypercube $Q_r$ satisfying
$F\circ \Phi_\ralpha = \Phi_\rbeta$. In particular, $\Phi_\ralpha$ is an
isometric embedding if and only if $\Phi_\rbeta$ is also an isometric embedding.
\end{lemma}

\begin{proof}
Define $F:\{0,1\}^r \to \{0,1\}^r$ via $F(\a) = \a + \Phi_\rbeta(\ralpha)$,
where $+$ denotes the bitwise XOR operation. It is clear that $F$ is an
automorphism of $Q_r$. Now, let $\rgamma\in [\ralpha]$. We need to prove that
$\Phi_\rbeta(\rgamma) = \Phi_\ralpha(\rgamma) + \Phi_\rbeta(\ralpha)$. Write
$\Phi_\rbeta (\rgamma) = a_1a_2\cdots a_r$, $\Phi_\ralpha(\rgamma) =
b_1b_2\cdots b_r$, and $\Phi_\rbeta(\ralpha) = c_1c_2\cdots c_r$. By definition
of the bitwise XOR operation, it suffices to show that $b_i + c_i \equiv a_i
\pmod{2}$ for all $i\in\{1,\ldots,r\}$. We proceed on a case-by-case basis.
	
First, suppose $b_i + c_i \equiv 0 \pmod 2$. Then $b_i = c_i$. If $b_i = 0
=c_i$, then $\supp_{\llb 2i\rrb}(\rgamma) = \supp_{\llb 2i\rrb}(\ralpha)$ and
$\supp_{\llb 2i\rrb}(\ralpha) = \supp_{\llb 2i\rrb}(\rbeta)$ by definitions of
$\Phi_{\ralpha}$ and $\Phi_{\rbeta}$, respectively. Hence $\supp_{\llb
2i\rrb}(\rgamma) = \supp_{\llb 2i\rrb}(\rbeta)$, so that $a_i = 0$, as well. On
the other hand, if $b_i =1= c_i$, then $\supp_{\llb 2i\rrb}(\rgamma) \neq
\supp_{\llb 2i\rrb}(\ralpha)$ and $\supp_{\llb 2i\rrb}(\ralpha) \neq
\supp_{\llb 2i\rrb}(\rbeta)$. It follows that $\supp_{\llb 2i\rrb}(\rgamma) =
\supp_{\llb 2i\rrb}(\rbeta)$, and so $a_i = 0$.
		
Next, suppose that  $b_i + c_i \equiv 1 \pmod 2$, so that $b_i \neq c_i$. If
$b_i = 0$ while $c_i = 1$, then $\supp_{\llb 2i\rrb}(\rgamma) = \supp_{\llb
2i\rrb}(\ralpha)$ while $\supp_{\llb 2i\rrb}(\ralpha) \neq  \supp_{\llb
2i\rrb}(\rbeta)$. This implies that $\supp_{\llb 2i\rrb}(\rgamma) \neq
\supp_{\llb 2i\rrb}(\rbeta)$, and so $a_i = 1$. Similarly, if $b_i=1$ while
$c_i =0$, then $\supp_{\llb 2i\rrb}(\rgamma) \neq \supp_{\llb 2i\rrb}(\ralpha)$
while $\supp_{\llb 2i\rrb}(\ralpha) =  \supp_{\llb 2i\rrb}(\rbeta)$. Thus,
$\supp_{\llb 2i\rrb}(\rgamma)\neq \supp_{\llb 2i\rrb}(\rbeta)$, so that $a_i =
1$.

This proves that $F$ is an automorphism of $Q_r$ having the desired property.
The second claim follows immediately.
\end{proof}

We now prove the following proposition, which shows that the braid graph for
every link is a partial cube with isometric dimension at most the rank of the
link.

\begin{proposition}\label{prop:braid graph for a link is a partial cube}
If $(W,S)$ is a simply-laced triangle-free Coxeter system and $\ralpha$ is a
link of rank $r\geq 0$, then the map $\Phi_{\ralpha}$ given in
Definition~\ref{def:embedding} is an isometric embedding of $B(\ralpha)$ into
$Q_r$. In particular, the braid graph for a link is a partial cube with
isometric dimension at most $r$.
\end{proposition}

\begin{proof}
We proceed by induction on $r$. The result is immediate if $r\in\{0,1\}$. Now,
suppose that $r\geq 2$ and assume that the map
$\Phi_\rbeta:[\rbeta]\to\{0,1\}^{r-1}$ is an isometric embedding  of
$B(\rbeta)$ into $Q_{r-1}$ for every link $\rbeta$ of rank $r-1$. Let $\ralpha$
be a link of rank $r$. Choose $\rsigma\in[\ralpha]$ such that $\llb
2r-3,2r-1\rrb,\llb 2r-1,2r+1\rrb \in \bs(\rsigma)$ according to
Lemma~\ref{lem:structure of overlapping braids}. Then by Lemma~\ref{lem:delete
last two letters from special link}, $\hatrsigma$ is a link of rank $r-1$.
Applying the inductive hypothesis allows us to conclude that
$\Phi_{\hatrsigma}:[\hatrsigma]\to \{0,1\}^{r-1}$ is an isometric embedding of
$B(\hatrsigma)$ into $Q_{r-1}$. Now, utilizing Lemma~\ref{lem:partition of
braid class for a link}, define $\Phi:[\ralpha]\to \{0,1\}^r$ via
\[
\Phi(\rbeta) := \begin{cases}
\Phi_{\hatrsigma}(\hatrbeta)0, & \text{if } \rbeta \in X_{\rsigma}\\
\Phi_{\hatrsigma}(\checkrbeta)1, & \text{if } \rbeta \in Y_{\rsigma}.
\end{cases}
\]
Note that the map $\Phi$ is appending a 0 or 1 to the appropriate binary string
depending on whether the input is in $X_{\rsigma}$ or $Y_{\rsigma}$,
respectively. It is not hard to see that $\Phi = \Phi_\rsigma$. According to
Lemma~\ref{lem:isometric embedding lemma}, it suffices to prove that $\Phi$ is
an isometric embedding. Suppose $\supp_{\llb 2r\rrb}([\ralpha]) = \{s,t\}$ and
$\rsigma_{\llb 2r\rrb}=s$, where $m(s,t)=3$. There are three cases to consider.

First, suppose $\rbeta,\rgamma \in X_\rsigma$. Then $\rbeta = \hatrbeta st$ and
$\rgamma =\hatrgamma st$ according to Lemma~\ref{lem:firstlastletter}(b). Then
we have
\begin{align*}
d_{B(\ralpha)}(\rbeta,\rgamma) & = d_{B(\ralpha)}\left(\hatrbeta st, \hatrgamma
st\right)\\
& = d_{B(\hatrsigma)}\left(\hatrbeta, \hatrgamma\right) &
(\text{Corollary~\ref{cor:subgraph of a link}\ref{cor:subgraph of a link a}})\\
& =
d_{Q_{r-1}}\left(\Phi_{\hatrsigma}(\hatrbeta),\Phi_{\hatrsigma}(\hatrgamma
)\right) & (\text{inductive hypothesis})\\
& =
d_{Q_{r}}\left(\Phi_{\hatrsigma}(\hatrbeta)0,\Phi_{\hatrsigma}(\hatrgamma
)0\right)\\
& = d_{Q_{r}}\left(\Phi_{\rsigma}(\rbeta), \Phi_{\rsigma}(\rgamma)\right).
\end{align*}
The case where $\rbeta,\rgamma \in Y_\rsigma$ is similar, except
Corollary~\ref{cor:subgraph of a link}\ref{cor:subgraph of a link b} is used
instead of Corollary~\ref{cor:subgraph of a link}\ref{cor:subgraph of a link a}
and we append a 1 in place of a 0. Lastly, suppose that $\rbeta\in X_\rsigma$
while $\rgamma \in Y_\rsigma$. Then $\rbeta =\hatrbeta st$ and $b_{\llb
2r-1,2r+1\rrb}(\rgamma) =\checkrgamma st$. Thus, we have
\begin{align*}
d_{B(\ralpha)}(\rbeta,\rgamma) & = d_{B(\ralpha)}\left(\rbeta,b_{\llb
2r-1,2r+1\rrb}(\rgamma)\right) + 1 & (\text{Corollary~\ref{cor:subgraph of a
link}\ref{cor:subgraph of a link c}})\\
& = d_{B(\ralpha)}\left(\hatrbeta st,\checkrgamma st\right) + 1  \\
& = d_{B(\hatrsigma)}\left(\hatrbeta,\checkrgamma \right) + 1 &
(\text{Corollary~\ref{cor:subgraph of a link}\ref{cor:subgraph of a link a}})\\
& = d_{Q_{r-1}}\left(\Phi_{\hatrsigma}(\hatrbeta),
\Phi_{\hatrsigma}(\checkrgamma )\right)+1 & (\text{inductive hypothesis})\\
& = d_{Q_{r}}\left(\Phi_{\hatrsigma}(\hatrbeta)0,
\Phi_{\hatrsigma}(\checkrgamma )1\right)\\
& = d_{Q_{r}}\left(\Phi_{\rsigma}(\rbeta), \Phi_{\rsigma}(\rgamma)\right).
\end{align*}
This verifies that $\Phi$ is an isometric embedding. It follows that
$B(\ralpha)$ is a partial cube with isometric dimension at most
$\rank(\ralpha)$.
\end{proof}

\begin{example}\label{ex:embed the lolli}
Consider the reduced expressions $\rgamma_1,\ldots,\rgamma_5$ given in
Example~\ref{ex:braid graphs}(c). The braid graph $B(\rgamma_1)$ was shown in
Figure~\ref{fig:braidgraphc}. In Example~\ref{ex:linkschains}, we showed that
$\rgamma_1$ is a link of rank $3$. Proposition~\ref{prop:braid graph for a link
is a partial cube} guarantees that there are at least five distinct isometric
embeddings of $B(\rgamma_1)$ into $Q_3$, one for each of $\rgamma_1$,
$\rgamma_2$, $\rgamma_3$, $\rgamma_4$, and $\rgamma_5$. One possible embedding,
using  $\rgamma_4$, is shown in Figure~\ref{fig:banner embedding}. It is clear
in this example that $\dim_I(B(\rgamma_1)) = 3 = \rank(\rgamma_1)$ since $Q_2$
has only four vertices.
\end{example} 

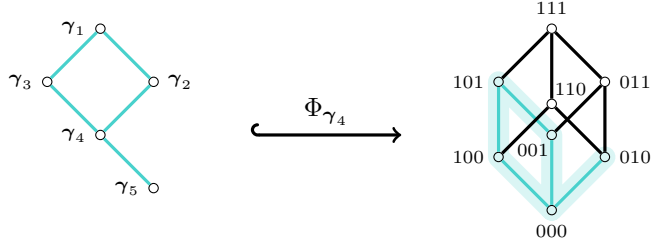
\begin{figure}[!ht]
	\centering
	{
\begin{tikzpicture}[every circle node/.style={draw, circle ,inner sep=1.25pt}]
			\node [circle] (1) [label=left:$\scriptstyle \rgamma_4$] at (0,1){};
			\node [circle] (2) [label=left:$\scriptstyle \rgamma_5$] at (.707,.293){};
			\node [circle] (3) [label=left:$\scriptstyle \rgamma_3$] at (-.707,1.707){};
			\node [circle] (4) [label=right:$\scriptstyle \rgamma_2$] at (.707,1.707){};
			\node [circle] (5) [label=left:$\scriptstyle\rgamma_1$] at (0,2.414){};
			\draw [turq,-, very thick] (1) to (2);
			\draw [turq,-, very thick] (1) to (3);
			\draw [turq,-, very thick] (1) to (4);
			\draw [turq,-, very thick] (3) to (5);
			\draw [turq,-, very thick] (5) to (4);
			\node [circle] (6) at (6,1){};
			\node [circle] (7) [label=left: $\scriptstyle 101$] at (5.293,1.707){};
			\node [circle] (8) [label=right:$\scriptstyle 011$] at (6.707,1.707){};
			\node [circle] (9) [label=above: $\scriptstyle 111$] at (6,2.414){};
			\node [circle] (10) [label=below: $\scriptstyle 000$] at (6,0){};
			\node [circle] (11) [label=left: $\scriptstyle 100$] at (5.293,0.707){};
			\node [circle] (12) [label=right: $\scriptstyle 010$] at (6.707,0.707){};
			\node [circle] (13) at (6,1.414){};
			\node at (6-.25,1-.2) {$\scriptstyle 001$};
			\node at (6+.25,1.414+.2) {$\scriptstyle 110$};
			\draw [turq,-, very thick] (6) to (7);
			\draw [black,-, very thick] (6) to (8);
			\draw [black,-, very thick] (7) to (9);
			\draw [black,-, very thick] (9) to (8);
			\draw [turq,-, very thick] (10) to (11);
			\draw [turq,-, very thick] (10) to (12);
			\draw [black,-, very thick] (11) to (13);
			\draw [black,-, very thick] (13) to (12);
			\draw [black,-, very thick] (13) to (9);
			\draw [turq,-, very thick] (6) to (10);
			\draw [turq,-, very thick] (7) to (11);
			\draw [black,-, very thick] (8) to (12);
			
\draw [very thick, right hook->,    black] (2,1) -- (4,1)
node[midway,above]{$\Phi_{\rgamma_4}$};
			
			\begin{pgfonlayer}{background}
\highlight{8pt}{turq}{(12.center) to (10.center) to (11.center) to (7.center)
to (6.center) to (10.center)}
			\end{pgfonlayer}
		
		\end{tikzpicture}
}\caption{An isometric embedding of $B(\rgamma_1)$ into $Q_3$ for
Example~\ref{ex:embed the lolli}.}\label{fig:banner embedding}
\end{figure}

We make the following conjecture regarding the isometric dimension of the braid
graph for a link.

\begin{conjecture}\label{conj:partial dimension of a link is its rank}
If $(W,S)$ is a simply-laced triangle-free Coxeter system and $\ralpha$ is a
link, then $\dim_I(B(\ralpha)) = \rank(\ralpha)$.
\end{conjecture}

Given any reduced expression in a simply-laced triangle-free Coxeter system, we
can apply Corollary~\ref{cor:boxproductsofbraidgraphs},
Proposition~\ref{prop:braid graph for a link is a partial cube}, and
Proposition~\ref{prop:boxprodpartial} to immediately obtain the following
theorem, which is the main result of this paper. Note that one can acquire the
appropriate labeling in terms of binary strings by applying
Proposition~\ref{prop:braid graph for a link is a partial cube} to each link
factor and then concatenating the resulting binary strings.

\begin{theorem}\label{thm:braid graphs are partial cubes}
If $(W,S)$ is a simply-laced triangle-free Coxeter system and $\ralpha$ is a
reduced expression for some $w\in W$, then there is an isometric embedding of
$B(\ralpha)$ into $Q_{\rank(\ralpha)}$. In particular, $B(\ralpha)$ is a
partial cube with isometric dimension at most $\rank(\ralpha)$.
\end{theorem}

Since the number of vertices in $Q_{\rank(\ralpha)}$ is $2^{\rank(\ralpha)}$,
we obtain the follow corollary.

\begin{corollary}\label{cor:braidclassbound}
If $(W,S)$ is a simply-laced triangle-free Coxeter system and $\ralpha$ is a
reduced expression for some $w\in W$, then $\card([\ralpha]) \leq
2^{\rank(\ralpha)}$. Moreover, if $\ralpha$ is a reduced expression with link
factorization $\ralpha_1\mid\ralpha_2\mid\cdots\mid\ralpha_k$, then the bound
is achieved precisely when $\rank(\ralpha_i) \leq 1$ for every $i$.
\end{corollary}

If Conjecture~\ref{conj:partial dimension of a link is its rank} is true, then
we immediately obtain the following conjecture by applying
Proposition~\ref{prop:boxprodpartial}.

\begin{conjecture}\label{conj:partial dimension of a reduced expression is sum
of ranks}
If $(W,S)$ is simply-laced triangle-free Coxeter system and $\ralpha$ is a
reduced expression with link factorization
$\ralpha_1\mid\ralpha_2\mid\cdots\mid\ralpha_k$, then
\[
\dim_I(B(\ralpha))= \sum_{i=1}^k \rank(\ralpha_i).
\]
\end{conjecture}

If $(W,S)$ is a simply-laced Coxeter system and $w\in W$ such that $\ell(w)=m$,
then the maximum number of braid shadows that any reduced expression for $w$
can have is $\lfloor \frac{m-1}{2}\rfloor$. For simply-laced Coxeter systems,
it is clear that as the length increases, so too does the number of possible
braid shadows. For finite Coxeter groups, the length function is bounded above
by the length of the longest element. It turns out that all of the finite
simply-laced Coxeter systems are triangle free~\cite{Humphreys1990}. Moreover,
every finite simply-laced Coxeter system is isomorphic to a direct product of
some combination of Coxeter systems of types $A_n$ ($n\geq 1$), $D_n$ ($n\geq
4$), $E_6$, $E_7$, and $E_8$. For each of these groups, we can utilize
Corollary~\ref{cor:braidclassbound} to obtain an upper bound on the cardinality
of any braid class in that group.

Since the length of the longest element in the Coxeter system of type $A_n$ is
$\frac{(n+1)n}{2}$, the maximum number of braid shadows that a reduced
expression for an element in $W(A_n)$ could have is $\lfloor
\frac{n^2+n-2}{4}\rfloor$.  This bound is attained at least when $n=1,2,3$.
Nonetheless, Corollary~\ref{cor:braidclassbound} implies that if $\ralpha$ is a
reduced expression for some element in $W(A_n)$, then $\card([\ralpha])\leq
2^{\lfloor \frac{n^2+n-2}{4}\rfloor}$. For example, if $n=3$, then the
cardinality of every braid class in $W(A_3)$ must be less than or equal to 4.
However, the maximum size of a braid class in $W(A_3)$ is 3. As $n$ increases,
our bound deteriorates. In~\cite{Zollinger1994a}, Zollinger establishes sharp
upper bounds on the cardinality of a braid class for a fixed length across all
Coxeter systems of type $A_n$.

In the case of type $D_n$, the longest element has length $n^2-n$, and so the
maximum number of braid shadows that a reduced expression for an element in
$W(D_n)$ could have is $\lfloor \frac{n^2-n-1}{2}\rfloor$. Thus, if $\ralpha$
is a reduced expression for some element in $W(D_n)$, then
$\card([\ralpha])\leq 2^{\lfloor \frac{n^2-n-1}{2}\rfloor}$.

\end{section}

\begin{section}{Fibonacci links and Fibonacci cubes}\label{sec:fibonacci}  

We now introduce a special class of links with connections to the Fibonacci
numbers. As usual, we define the Fibonacci numbers via $F_n=F_{n-1}+F_{n-2}$
with $F_1=F_2=1$.

\begin{definition}
Suppose $(W,S)$ is a simply-laced triangle-free Coxeter system. If $\rphi$ is a
link with the property that $\bs([\rphi]) = \bs(\rphi)$, then we will refer to
$\rphi$ as a \emph{Fibonacci link} and the corresponding braid class $[\rphi]$
will be referred to as a \emph{Fibonacci chain}.
\end{definition}

By invoking Lemma~\ref{lem:delete last two letters from special link}, we see
that if $\rphi$ is a Fibonacci link of rank $r\geq 2$, then $\hatrphi$ and
$\doublehatrphi$ are Fibonacci links of ranks $r-1$ and $r-2$, respectively.
The following proposition describes the connection between a Fibonacci chain
and the Coxeter graph. Recall that the \emph{star graph} $S_k$ with one
internal node and $k$ leaves is defined to be the complete bipartite graph
$K_{1,k}$.

\begin{proposition}\label{prop:fibonacci iff star graph}
Suppose $(W,S)$ is a simply-laced triangle-free Coxeter system of type $\Gamma$
and let $\ralpha$ be a link of rank at least one. Then $\ralpha$ is braid
equivalent to a Fibonacci link if and only if the induced subgraph
$\Gamma[\supp(\ralpha)]$ is a star graph.
\end{proposition}

\begin{proof}
Suppose that $\rphi \in [\ralpha]$ is a Fibonacci link and write $\rphi =
s_{x_1}s_{x_2} \cdots s_{x_{2r+1}}$. By definition, $\llb 2i-1,2i+1 \rrb \in
\bs(\rphi)$ for all $i\in \{1,2,\ldots,r\}$, which implies that the letters in
each factor $\rphi_{\llb 2i-1,2i+1\rrb}$ satisfy $s_{x_{2i-1}}=s_{x_{2i+1}}$
and $m(s_{x_{2i-1}},s_{x_{2i}}) = 3=m(s_{x_{2i}},s_{x_{2i+1}})$. If we write
$s:=s_{x_1}$, then we have $s = s_{x_{2i+1}}$ for all $i\in \{1,2,\ldots,r\}$,
so that $\supp(\ralpha) = \supp(\rphi) =
\{s,s_{x_2},s_{x_4},\ldots,s_{x_{2r}}\}$, possibly with repeats among
$s_{x_2},s_{x_4},\ldots,s_{x_{2r}}$. Notice that $m(s_{x_{2k}},s_{x_{2j}}) = 2$
whenever $s_{x_{2k}}\neq s_{x_{2j}}$ since $m(s,s_{x_{2k}})= 3 =
m(s,s_{x_{2j}})$ and $\Gamma$ has no three-cycles. Thus, the subgraph of
$\Gamma$ induced by $\supp(\ralpha)$ is a star graph with
$\card(\supp(\ralpha))-1\leq r$ leaves.

We will prove the converse by induction. If $\ralpha$ is a link of rank 1, then
$\ralpha = sts$ for some $s,t \in S$ with $m(s,t) = 3$, in which case,
$\ralpha$ is a Fibonacci link. Now, assume that the statement holds for all
links of rank $r\geq 1$ and let $\ralpha$ be a link of rank $r+1$ such that the
induced subgraph $\Gamma[\supp(\ralpha)]$ is a star graph. Then we may assume
that $\supp(\ralpha) = \{s,t_1,t_2,\ldots,t_k\}$ for some $k$, where $m(s,t_i)
= 3$ for all $i$ and $m(t_i,t_j) = 2$ for all $i \neq j$. According to
Lemma~\ref{lem:structure of overlapping braids}, choose a link $\rsigma \in
[\ralpha]$ with the property that $\llb 2r-1,2r+1\rrb,\llb 2r+1,2r+3\rrb \in
\bs(\rsigma)$. Suppose $\rsigma_{\llb 2r-1,2r+1\rrb}=st_ks$. By
Lemma~\ref{lem:delete last two letters from special link}, $\hatrsigma$ is a
link of rank $r$ with $\supp(\hatrsigma) \subseteq \supp(\rsigma)$, so that
$\Gamma[\supp(\hatrsigma)]$ is also a star graph. By the inductive hypothesis,
there exists a Fibonacci link $\rpsi\in [\hatrsigma]$.  But then by
Lemma~\ref{lem:partition of braid class for a link}(b), there exists $\rphi\in
X_{\rsigma}$ such that $\hatrphi=\rpsi$.  This implies that $\rphi=\rpsi t_ks$,
and since $\rpsi$ is a Fibonacci link ending in $s$, it follows that $\rphi$ is
a Fibonacci link braid equivalent to $\ralpha$.
\end{proof} 

Implicit in the proof of the preceding proposition is the fact that every
Fibonacci link of rank $r\geq 1$ can be written in the form
\[
\rphi = st_{1}st_{2}\cdots st_{r-1}st_{r}s
\]
for some $s,t_{1},t_{2},\ldots,t_{r}\in S$, possibly with repeats among
$t_{1},t_{2},\ldots,t_{r}$, where $m(s,t_{i}) = 3$ for all $i \in
\{1,2\ldots,r\}$ and $m(t_{i},t_{j}) =2$ for $t_i\neq t_j$.  We will use this
fact frequently in what follows.

\begin{example}\label{ex:not a fibonacci link}
Not every link is a Fibonacci link. For instance, in the Coxeter system of type
$A_4$, the reduced expression $1213243$ is a link, but not a Fibonacci link.
Similarly, in the Coxeter system of type $D_5$, the reduced expression
$4534313234313$ is a link, but it is not a Fibonacci link. Moreover, neither
link is braid equivalent to a Fibonacci link according to
Proposition~\ref{prop:fibonacci iff star graph}.
\end{example}

\begin{example}\label{ex: fibonacci links in A and D} 
In the Coxeter system of type $A_n$, there are no Fibonacci links of rank 3 or
larger. This follows from Proposition~\ref{prop:braid graph for string} or
Lemma~1 from~\cite{Zollinger1994a}.  The Fibonacci links of rank 0 are
precisely the reduced expressions consisting of a single generator. The
Fibonacci links of ranks 1 and 2 are of the form $sts$ and $tstut$,
respectively, where $m(s,t)=3=m(t,u)$ and $m(s,u)=2$.
\end{example}

\begin{example}
It is more difficult to write down all the Fibonacci links in the Coxeter
system of type $D_n$. The following reduced expressions are Fibonacci links of
ranks 1 through 5:
\[
343,\, 34313,\, 3431323, \, 343132343, \, 34313234313.
\]
In fact, it is not possible to find a Fibonacci link of rank greater than $5$
in type $D_n$. This fact is obvious in type $D_4$ since the longest element has
length 12, but remains true even for large $n$. Notice that
Proposition~\ref{prop:fibonacci iff star graph} implies that every link in type
$D_4$ is braid equivalent to a Fibonacci link since the Coxeter graph of type
$D_4$ is a star graph.
\end{example}

In his MS thesis~\cite{Cadman2021}, Cadman proved that Fibonacci links of
arbitrary rank occur in any simply-laced triangle-free Coxeter system whose
corresponding Coxeter graph contains the Coxeter graph of type
$\widetilde{D}_4$ as a subgraph. The crux of Cadman's construction is proving
that the proposed class of expressions is reduced. His argument relied heavily
on technical results involving root sequences.

The following definition is similar to Definition~\ref{def:braidgraph
decomposition embeddings for a link}.

\begin{definition}
Suppose $(W,S)$ is a simply-laced triangle-free Coxeter system and $\rphi
=st_1st_2\cdots st_{r-1}st_rs$ is a Fibonacci link of rank $r\geq 2$. Denote by
$\Sigma:[\doublehatrphi] \to [\rphi]$ the function that conjoins to each
element of $[\doublehatrphi]$ the letters $t_{r-1}t_rst_r$ on the right.
\end{definition}

Note that the function $\Sigma:[\doublehatrphi] \to [\rphi]$ is well defined by
two applications of Lemma~\ref{lem:partition of braid class for a link}. Each
Fibonacci link carries with it rich recursive properties that manifest in the
structure of the braid class and the corresponding braid graph. The next result
summarizes these properties.

\begin{proposition}\label{prop:fibonaccipartitionnew}
If $(W,S)$ is a simply-laced triangle-free Coxeter system and $\rphi
=st_1st_2\cdots st_{r-1}st_rs$ is a Fibonacci link of rank $r\geq 2$, then the
function $\Sigma:[\doublehatrphi] \to [\rphi]$ is an isometric embedding from
$B(\doublehatrphi)$ into $B(\rphi)$. In particular, $\im(\Sigma) = Y_\rphi$ and
$B(\doublehatrphi)$ is isomorphic to the induced subgraph $B(\rphi)[Y_\rphi]$.
\end{proposition}

\begin{proof}
Using Corollary~\ref{cor:subgraph of a link}\ref{cor:subgraph of a link b}, it
suffices to prove that $\im(\Sigma) = Y_\rphi$. It is clear from the
definitions that $\im(\Sigma) \subseteq Y_\rphi$. Let $\rbeta \in Y_\rphi$.
Then $\llb 2r-1,2r+1\rrb \in S(\rbeta)$ and $\checkrbeta \in [\hatrphi]$
according to Lemma~\ref{lem:partition of braid class for a link}(c). Moreover,
the proof of Lemma~\ref{lem:partition of braid class for a link}(c) shows that
$\rbeta_{\llb 2r-2 \rrb} = \rphi_{\llb 2r-2 \rrb}$ so that $\checkrbeta \in
X_\hatrphi$. Hence, $\hatcheckrbeta \in [\doublehatrphi]$. But $\hatcheckrbeta$
is obtained from $\rbeta$ by deleting the rightmost four letters so that
$\Sigma\left(\hatcheckrbeta\right) = \rbeta$.
\end{proof}

The preceding proposition is a strengthening of Corollary~\ref{cor:subgraph of
a link}\ref{cor:subgraph of a link b} for Fibonacci links. By combining
Corollary~\ref{cor:subgraph of a link} with
Proposition~\ref{prop:fibonaccipartitionnew}, we can conclude that the braid
graph for a Fibonacci link $\rphi$ of rank $r\geq 2$ is, in some sense, a
gluing together of the braid graphs for a Fibonacci link of rank $r-1$ and a
Fibonacci link of rank $r-2$. The gluing of these two isometric subgraphs
corresponds to applying the braid move in the braid shadow $\llb 2r-1,2r+1\rrb$
to any element of $Y_\rphi$. A concrete example is given below.

\begin{example}\label{ex:fibonaccipartition}
Consider the Coxeter system of type $D_4$. The reduced expression $\rphi =
343132343$ is a Fibonacci link of rank 4. The partition of $[\rphi]$ from
Lemma~\ref{lem:partition of braid class for a link} is given by
\[
X_\rphi = \{343132343,434132343,434123243,341312343,343123243\}
\]
and 
\[
Y_\rphi = \{343132434,434132434,341312434\}.
\]
The first block has $F_5 = 5$ elements and is in bijection with $[\hatrphi]$,
while the second block has $F_4 = 3$ elements and is in bijection with
$\left[\doublehatrphi\right]$. The braid graph for $\rphi$ is shown in
Figure~\ref{fig:fibonaccicubepartition}. The isometric subgraphs
$B(\rphi)[X_\rphi]$ and $B(\rphi)[Y_\rphi]$ are highlighted in
\textcolor{green}{green} and \textcolor{magenta}{magenta}, and are isomorphic
to $B(\hatrphi)$ and $B(\doublehatrphi)$, respectively. Similar to
Example~\ref{ex:braid graph with partition}, each of the edges joining
$B(\rphi)[X_{\rphi}]$ and $B(\rphi)[Y_{\rphi}]$ correspond to the braid move
applied in the rightmost braid shadow.
\end{example}

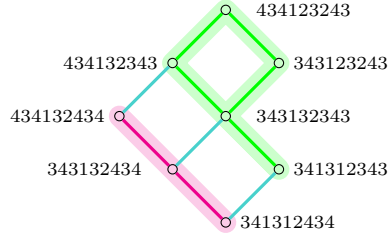
\begin{figure}[h]
\begin{tikzpicture}[every circle node/.style={draw, circle ,inner sep=1.25pt}]

\node [circle] (1) [label=right:\; $\scriptstyle 343132343$] at (0,1){};
\node [circle] (2) [label=right: $\scriptstyle 341312343$] at (.707,.293){};
\node [circle] (3) [label=left: $\scriptstyle 434132343$] at (-.707,1.707){};
\node [circle] (4) [label=right: $\scriptstyle 343123243$] at (.707,1.707){};
\node [circle] (5) [label=right:\;  $\scriptstyle 434123243$] at (0,2.414){};
\node [circle] (6) [label=left: $\scriptstyle 343132434$ \;] at (-.707,.293){};
\node [circle] (7) [label=right: $\scriptstyle 341312434$] at (0,-0.414){};
\node [circle] (8) [label=left: $\scriptstyle  434132434$] at (-1.414,1){};
\draw [green,-, very thick] (1) to (2);
\draw [green,-, very thick] (1) to (3);
\draw [green,-, very thick] (1) to (4);
\draw [turq,-, very thick] (1) to (6);
\draw [turq,-, very thick] (8) to (3);
\draw [magenta,-, very thick] (8) to (6);
\draw [green,-, very thick] (3) to (5);
\draw [magenta,-, very thick] (6) to (7);
\draw [green,-, very thick] (5) to (4);
\draw [turq,-, very thick] (7) to (2);
\begin{pgfonlayer}{background}
\highlight{8pt}{green}{(1.center) to (2.center) to (3.center) to (5.center) to
(4.center) to (1.center)}
\highlight{8pt}{magenta}{(7.center) to (8.center)}
\end{pgfonlayer}
\end{tikzpicture} 
\caption{Braid graph for the Fibonacci link in
Example~\ref{ex:fibonaccipartition} together with a partition of the vertices
according to Lemma~\ref{lem:partition of braid class for a
link}.}\label{fig:fibonaccicubepartition}
\end{figure}

Our next objective is to prove that Fibonacci chains contain a Fibonacci number
of links.

\begin{proposition}\label{prop:fibonaccilinks}
Suppose $(W,S)$ is a simply-laced triangle-free Coxeter system. If $\rphi$ is a
Fibonacci link of rank $r \geq 0$, then $\card([\rphi]) = F_{r+2}$.
\end{proposition}

\begin{proof}
The proof is by induction on $r$. The case with $r=0$ is immediate. If $r=1$,
then $\rphi = sts$ for some $s,t \in S$ with $m(s,t) = 3$. In this case,
$\card([\rphi]) = \card(\{sts,tst\}) = 2 = F_3$. Now, let $r\geq 1$ and assume
that every Fibonacci link of rank $k\leq r$ has $F_{k+2}$ elements in its braid
class. Let $\rphi$ be a Fibonacci link of rank $r+1$. Induction together with
Lemma~\ref{lem:partition of braid class for a link}(a) and
Proposition~\ref{prop:fibonaccipartitionnew} implies that
\[
\card([\rphi]) = \card\left(\left[\doublehatrphi\right]\right) +
\card([\hatrphi])= F_{r+1} + F_{r+2} = F_{r+3},
\]
which proves the claim.
\end{proof}

Next, we show that the braid graph for a Fibonacci link of rank $r$ is
isomorphic to a well-known isometric subgraph of $Q_r$. We define the
\emph{Fibonacci cube} of order $r$ as the subgraph $\F_r := Q_r[V_r]$ induced
by set of vertices
\[
V_r = \{a_1a_2\cdots a_r \in \{0,1\}^r\mid a_i a_{i+1} = 0, \ 1 \leq i \leq
r-1\}.
\]
That is, $V_r$ is the collection of length $r$ binary strings that do not
contain the consecutive substring $11$. Fibonacci cubes have been studied
extensively in the literature and were introduced as a model for
interconnection networks~\cite{Cong1993,Hsu1993}. According
to~\cite{klavvzar2013structure}, the Fibonacci cube $\F_{r}$ is a partial cube
consisting of $F_{r+2}$ many vertices.

\begin{theorem}\label{thm:FibonacciCube}
Suppose $(W,S)$ is a simply-laced triangle-free Coxeter system. If $\rphi$ is a
Fibonacci link of rank $r$, then the braid graph $B(\rphi)$ is isomorphic to
the Fibonacci cube $\F_r$.
\end{theorem}

\begin{proof}
According to Proposition~\ref{prop:braid graph for a link is a partial cube},
the map $\Phi_\rphi:[\rphi] \to \{0,1\}^r$ is an isometric embedding of
$B(\rphi)$ into $Q_r$. Thus, it suffices to show that $\im\left(\Phi_\rphi
\right) = V_r$. Using Proposition~\ref{prop:fibonaccilinks}, we already know
that $\card([\rphi]) = F_{r+2} = \card(V_r)$. Therefore, we only need to show
that $\im\left(\Phi_\rphi \right) \subseteq V_r$. We proceed by induction on
$r$. The claim is trivial if $r=0$ or $r=1$. Suppose that $r\geq 2$ and assume
the statement is true for every Fibonacci link of rank $k\leq r$. Let $\rphi$
be a link of rank $r+1$ and let $a_1a_2 \cdots a_ra_{r+1} \in \im(\Phi_\rphi)$.
Choose $\rbeta \in [\rphi]$ such that $\Phi_\rphi(\rbeta) = a_1a_2 \cdots
a_ra_{r+1}$. Recall from the proof of Proposition~\ref{prop:braid graph for a
link is a partial cube} that $\Phi_\rphi(\rbeta) = \Phi_{\hatrphi}(\hatrbeta)0$
if $\rbeta \in X_{\rphi}$ and $\Phi_\rphi(\rbeta) =
\Phi_{\hatrphi}\left(\checkrbeta\right)1$ if $\rbeta \in Y_{\rphi}$. In fact,
in the latter case it is easy to see that $\supp_{\llb 2r-2\rrb}(\rbeta) =
\supp_{\llb 2r-2\rrb}(\rphi)$ since $\llb 2r-3,2r-1\rrb, \llb 2r-1,2r+1 \rrb
\in\bs(\rphi)$ while $\rbeta$ and $\rphi$ differ by a single braid move in the
braid shadow $\llb 2r-1,2r+1 \rrb$. Thus, $\Phi_\rphi(\rbeta) =
\Phi_{\doublehatrphi}\left(\hatcheckrbeta\right)01$ if $\rbeta \in Y_\rphi$.
Since $\hatrphi$ and $\doublehatrphi$ are links of rank $r$ and $r-1$,
respectively, the inductive hypothesis implies that $\Phi_{\hatrphi}(\hatrbeta)
\in V_r$ or $\Phi_{\doublehatrphi}\left(\hatcheckrbeta\right) \in V_{r-1}$,
depending on whether $\rbeta \in X_\rphi$ or $\rbeta \in Y_\rphi$. Then it is
clear in either case that $\Phi_\rphi(\rbeta) \in V_{r+1}$ since either
$\Phi_\rphi(\rbeta) = \Phi_{\hatrphi}(\hatrbeta)0$ or $\Phi_\rphi(\rbeta) =
\Phi_{\doublehatrphi}\left(\hatcheckrbeta\right)01$. This completes the proof.
\end{proof}

The previous theorem together with \cite[Proposition~6.1]{Cabello2011} implies
that Conjecture~\ref{conj:partial dimension of a link is its rank} is settled
for Fibonacci links.

\begin{corollary}
If $(W,S)$ is a simply-laced triangle-free Coxeter system and $\rphi$ is a
Fibonacci link, then $\dim_I(B(\rphi)) = \rank(\rphi)$.
\end{corollary}

According to \cite[Corollary 4.2]{klavvzar2013structure}, a Fibonacci cube of
dimension $r\geq 2$ has a unique vertex of degree $r$. Certainly this vertex
attains the maximum degree in $\F_r$ since every vertex of $Q_r$ has degree
$r$. If $\rphi$ is a Fibonacci link of rank $r\geq 2$, it must be the case that
the vertex corresponding to $\rphi$ is the vertex that attains the maximum
degree $r$.  This yields the following corollary.

\begin{corollary}
If $(W,S)$ is a simply-laced triangle-free Coxeter system and $\rphi$ is a
Fibonacci link of rank at least two, then $\rphi$ is the unique Fibonacci link
in the Fibonacci chain $[\rphi]$.
\end{corollary}

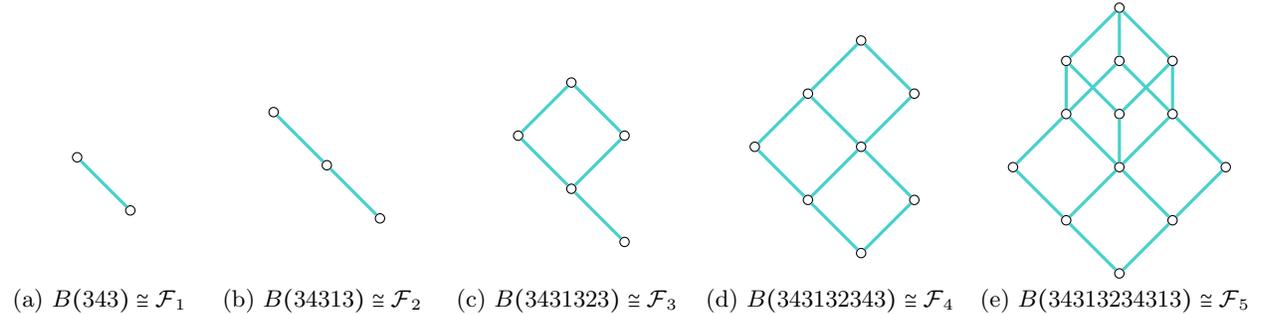
\begin{figure}[ht!]
 \subfloat[$B(343) \cong \F_1$]{
	\begin{minipage}[c][1\width]{
	   0.14\textwidth}
	   \centering
\begin{tikzpicture}[every circle node/.style={draw, circle,inner sep=1.25pt}]
\node [circle] (1) at (0,1){};
\node [circle] (2) at (.707,.293){};
\draw [turq,-, very thick] (1) to (2);

\end{tikzpicture}
	\end{minipage}}
 \hfill 	
  \subfloat[$B(34313)\cong \F_2$]{
	\begin{minipage}[c][1\width]{
	   0.17\textwidth}
	   \centering
\begin{tikzpicture}[every circle node/.style={draw, circle, inner sep=1.25pt}]
\node [circle] (1) at (0,1){};
\node [circle] (2) at (.707,.293){};
\node [circle] (3) at (-.707,1.707){};
\draw [turq,-, very thick] (1) to (2);
\draw [turq,-, very thick] (1) to (3);
\end{tikzpicture}
	\end{minipage}}
 \hfill 	
  \subfloat[$B(3431323) \cong \F_3$]{
	\begin{minipage}[c][1\width]{
	   0.175\textwidth}
	   \centering
\begin{tikzpicture}[every circle node/.style={draw, circle, inner sep=1.25pt}]
\node [circle] (1) at (0,1){};
\node [circle] (2) at (.707,.293){};
\node [circle] (3) at (-.707,1.707){};
\node [circle] (4) at (.707,1.707){};
\node [circle] (5) at (0,2.414){};
\draw [turq,-, very thick] (1) to (2);
\draw [turq,-, very thick] (1) to (3);
\draw [turq,-, very thick] (1) to (4);
\draw [turq,-, very thick] (3) to (5);
\draw [turq,-, very thick] (5) to (4);

\end{tikzpicture}
	\end{minipage}}
 \hfill 	
  \subfloat[$B(343132343) \cong \F_4$]{
	\begin{minipage}[c][1\width]{
	   0.2\textwidth}
	   \centering
\begin{tikzpicture}[every circle node/.style={draw, circle ,inner sep=1.25pt}]

\node [circle] (1) at (0,1){};
\node [circle] (2) at (.707,.293){};
\node [circle] (3) at (-.707,1.707){};
\node [circle] (4) at (.707,1.707){};
\node [circle] (5) at (0,2.414){};
\node [circle] (6) at (-.707,.293){};
\node [circle] (7) at (0,-0.414){};
\node [circle] (8) at (-1.414,1){};
\draw [turq,-, very thick] (1) to (2);
\draw [turq,-, very thick] (1) to (3);
\draw [turq,-, very thick] (1) to (4);
\draw [turq,-, very thick] (1) to (6);
\draw [turq,-, very thick] (8) to (3);
\draw [turq,-, very thick] (8) to (6);
\draw [turq,-, very thick] (3) to (5);
\draw [turq,-, very thick] (6) to (7);
\draw [turq,-, very thick] (5) to (4);
\draw [turq,-, very thick] (7) to (2);

\end{tikzpicture}
	\end{minipage}}
 \hfill	
  \subfloat[$B(34313234313)\cong \F_5$]{
	\begin{minipage}[c][1\width]{
	   0.21\textwidth}
	   \centering
\begin{tikzpicture}[every circle node/.style={draw, circle, inner sep=1.25pt}]
\node [circle] (1) at (0,1){};
\node [circle] (2) at (.707,.293){};
\node [circle] (3) at (-.707,1.707){};
\node [circle] (4) at (.707,1.707){};
\node [circle] (5) at (0,2.414){};
\node [circle] (6) at (-.707,.293){};
\node [circle] (7) at (0,-0.414){};
\node [circle] (8) at (-1.414,1){};
\draw [rotate = 90, turq,-, very thick] (1) to (2);
\draw [turq,-, very thick] (1) to (3);
\draw [turq,-, very thick] (1) to (4);
\draw [turq,-, very thick] (1) to (6);
\draw [turq,-, very thick] (8) to (3);
\draw [turq,-, very thick] (8) to (6);
\draw [turq,-, very thick] (3) to (5);
\draw [turq,-, very thick] (6) to (7);
\draw [turq,-, very thick] (5) to (4);
\draw [turq,-, very thick] (7) to (2);
\node [circle] (9) at (1.414,1){};

\node [circle] (10) at (0,1.707){};
\node [circle] (11) at (0,3.121){};
\node [circle] (13) at (.707,2.414){};
\node [circle] (14) at (-.707,2.414){};
\draw [turq,-, very thick] (1) to (10);
\draw [turq,-, very thick] (10) to (13);
\draw [turq,-, very thick] (10) to (14);
\draw [turq,-, very thick] (14) to (11);
\draw [turq,-, very thick] (13) to (11);
\draw [turq,-, very thick] (13) to (4);
\draw [turq,-, very thick] (4) to (9);
\draw [turq,-, very thick] (9) to (2);
\draw [turq,-, very thick] (5) to (11);
\draw [turq,-, very thick] (14) to (3);
\end{tikzpicture}
	\end{minipage}}
\caption{Braid graphs for several Fibonacci links in the Coxeter system of type
$D_4$. }
\label{fig:Fibonaccicubes}
\end{figure}

\begin{example}\label{ex:fibonaccicubes}
Consider the Fibonacci links in the Coxeter system of type $D_4$ from
Example~\ref{ex: fibonacci links in A and D}:
\[
343,\, 34313,\, 3431323, \, 343132343, \, 34313234313.
\]
According to Theorem~\ref{thm:FibonacciCube}, the corresponding braid graphs
are Fibonacci cubes. Each braid graph is depicted in
Figure~\ref{fig:Fibonaccicubes}. Note that for rank 2 and larger, the Fibonacci
link always corresponds to the vertex of highest degree.
\end{example}

The Fibonacci cube is not the only interesting graph that arises as the braid
graph for a reduced expression. For instance, the \emph{matchable Lucas cubes}
introduced in~\cite{wang2020structure} can be found in the Coxeter system of
type $D_5$ as seen in the following example. The matchable Lucas cubes are very
similar to the Fibonacci cubes, for example, the number of vertices in the
$n$th matchable Lucas cube is equal to the Lucas number $L_n$ defined by $L_n =
L_{n-1} + L_{n-2}$ with initial conditions $L_0 = 2$ and $L_1 = 1$.

\begin{example}\label{ex:lucas cubes}
The braid graph for the reduced expression $4534313234313$ in the Coxeter
system of type $D_5$ from Example~\ref{ex:not a fibonacci link}  is depicted in
Figure~\ref{fig:matchable lucas cube}. It turns out that this graph is
isomorphic to the $6$th matchable Lucas cube as introduced
in~\cite{wang2020structure}. In fact, it is not hard to see that the braid
graph for $4534313234313$ is a gluing together of the braid graphs for the
reduced expressions $453431323$ and $45343132343$. The braid graph for
$453431323$ is highlighted in \textcolor{magenta}{magenta} and is isomorphic to
the $4$th matchable Lucas cube, while the braid graph for $45343132343$ is
highlighted in \textcolor{green}{green} and is isomorphic to the $5$th
matchable Lucas cube. Each of the edges joining the braid graphs highlighted in
\textcolor{magenta}{magenta} and \textcolor{green}{green} correspond to the
braid move applied in rightmost braid shadow. Compare this example with
Examples~\ref{ex:braid graph with partition} and~\ref{ex:fibonaccipartition}.
\end{example}

\begin{figure}[h!]
\begin{tikzpicture}[every circle node/.style={draw, circle, inner sep=1.25pt},
xscale=-1]

\node [circle] (1) at (0,1){};
\node [circle] (2) at (.707,.293){};
\node [circle] (3) at (-.707,1.707){};
\node [circle] (4) at (.707,1.707){};
\node [circle] (5) at (0,2.414){};
\node [circle] (6) at (-.707,.293){};
\node [circle] (7) at (0,-0.414){};
\node [circle] (8) at (-1.414,1){};
\draw [green,-, very thick] (1) to (2);
\draw [turq,-,  very thick] (1) to (3);
\draw [green,-, very thick] (1) to (4);
\draw [green,-, very thick] (1) to (6);
\draw [magenta,-, very thick] (8) to (3);
\draw [turq,-, very thick] (8) to (6);
\draw [magenta,-, very thick] (3) to (5);
\draw [green,-, very thick] (6) to (7);
\draw [turq,-,  very thick] (5) to (4);
\draw [green,-, very thick] (7) to (2);
\node [circle] (9) at (1.414,1){};

\node [circle] (10) at (0,1.707){};
\node [circle] (11) at (0,3.121){};
\node [circle] (13) at (.707,2.414){};
\node [circle] (14) at (-.707,2.414){};
\draw [green,-, very thick] (1) to (10);
\draw [green,-, very thick] (10) to (13);
\draw [turq,-,  very thick] (10) to (14);
\draw [magenta,-, very thick] (14) to (11);
\draw [turq,-,  very thick] (13) to (11);
\draw [green,-, very thick] (13) to (4);
\draw [green,-, very thick] (4) to (9);
\draw [green,-, very thick] (9) to (2);
\draw [magenta,-, very thick] (5) to (11);
\draw [magenta,-, very thick] (14) to (3);

\node [circle] (15) at (1.414+.707,1+.707){};

\node [circle] (16) at (0+.707,2.414+.707){};
\node [circle] (17) at (0+.707,3.121+.707){};
\node [circle] (18) at (.707+.707,2.414+.707){};
\node [circle] (19) at (1.414,2.414){};

\draw [magenta,-, very thick] (5) to (16);
\draw [magenta,-, very thick] (11) to (17);
\draw [magenta,-, very thick] (16) to (17);
\draw [turq,-,very thick] (17) to (18);
\draw [green,-, very thick] (18) to (19);
\draw [green,-, very thick] (19) to (15);
\draw [turq,-, very thick] (19) to (16);
\draw [green,-, very thick] (19) to (4);
\draw [green,-, very thick] (15) to (9);
\draw [green,-, very thick] (18) to (13);

\begin{pgfonlayer}{background}
\highlight{8pt}{green}{(6.center) to (7.center) to  (2.center) to (9.center) to
(15.center)  to (19.center) to  (18.center)   to (13.center) to (10.center)}
\highlight{8pt}{green}{(6.center) to (1.center) to  (4.center) to (19.center) }
	\highlight{8pt}{green}{(9.center) to (4.center) to  (13.center) }
	\highlight{8pt}{green}{(2.center) to (1.center) to  (10.center)  }
\highlight{8pt}{magenta}{(8.center) to (3.center) to  (5.center) to (16.center)
to  (17.center)  to (11.center) to  (14.center)   to (3.center)}
	\highlight{8pt}{magenta}{(11.center) to (5.center)}
\end{pgfonlayer}

\node (h) [label = left: $\scriptstyle 4534313234313$ ] at (.75,0){};
\node (y) at (0,1){};
\node (x) at (1,0){};
\draw [->, thick] (x) to [out=180,in=270] (y);
\node (i) [label = right:  $\scriptstyle 4534313234131$] at (-1.25,2.1){};
\node (z) at (-.707,1.707){};
\node (w) at (-1.5,2.1){};
\draw [->, thick] (w) to [out=0,in=145]  (z);
\end{tikzpicture}
\caption{Braid graph for the reduced expression in Example~\ref{ex:lucas
cubes}.}\label{fig:matchable lucas cube}
\end{figure}
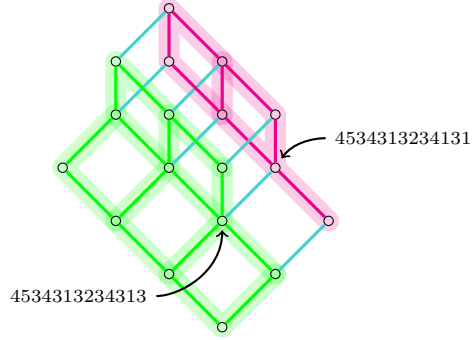

\end{section}

\begin{section}{Closing}\label{sec:closing}

There are several open problems and potential natural generalizations that
arise from this paper that we outline below.

\begin{enumerate}[(1)]

\item Theorem~\ref{thm:braid graphs are partial cubes} (and its counterpart for
links, Proposition~\ref{prop:braid graph for a link is a partial cube}) states
that if $(W,S)$ is a simply-laced triangle-free Coxeter system and $\ralpha$ is
a reduced expression for some $w\in W$, $B(\ralpha)$ is a partial cube with
isometric dimension at most $\rank(\ralpha)$. We conjecture that the isometric
dimension is exactly $\rank(\ralpha)$ (see Conjectures~\ref{conj:partial
dimension of a link is its rank} and~\ref{conj:partial dimension of a reduced
expression is sum of ranks}).

\item There is an intimate connection between partial cubes and the so-called
Djokovi\'c--Winkler relation on the collection of edges
(see~\cite{ovchinnikov2008partial}). It turns out that since braid graphs for
reduced expressions in simply-laced triangle-free Coxeter systems are partial
cubes, the Djokovi\'c--Winkler relation on set of edges of the braid graph is
an equivalence relation.  We conjecture that the corresponding equivalence
classes are the collections of edges that correspond to specific braid shadows.
If this is true, it would settle the conjecture in the previous item.

\item Suppose $(W,S)$ is a simply-laced triangle-free Coxeter system and
$\ralpha$ is a link of rank $r\geq 2$ and assume the notation of
Definition~\ref{def:braidgraph decomposition embeddings for a link}. In
Corollary~\ref{cor:subgraph of a link}, we state that the induced subgraph
$B(\ralpha)[Y_\rsigma]$ is an isometric subgraph of $B(\ralpha)$. In the spirit
of Proposition~\ref{prop:fibonaccipartitionnew}, we conjecture that
$B(\ralpha)[Y_\rsigma]$ is also a braid graph for some link derived from a link
in $Y_\rsigma$.

\item Can we strengthen the conclusion of Theorem~\ref{thm:braid graphs are
partial cubes}? We conjecture that braid graphs in simply-laced triangle-free
Coxeter systems are median graphs. Since Fibonacci cubes are median graphs
(see~\cite[Theorem~1]{klavvzar2005}), this conjecture is true for Fibonacci
links. In fact, every braid graph exhibited in this paper is a median graph.

\item We conjecture that the conclusion of Theorem~\ref{thm:braid graphs are
partial cubes} holds for all simply-laced Coxeter systems, not just those that
are triangle-free.

\item In Section~\ref{sec:type A}, we provide a classification of braid graphs
in Coxeter systems of type $A_n$. It would be interesting to provide an
analogous classification of braid graphs in other simply-laced arbitrary
Coxeter system (e.g., types $D_n$, $\widetilde{D}_n$, and $\widetilde{A}_n$).

\item Given a link in a simply-laced triangle-free Coxeter system, what is the
image of the corresponding braid class under the map described in
Definition~\ref{def:embedding}?

\item Call a subset $X \subseteq\{0,1\}^n$ \emph{admissible} if there exists a
simply-laced triangle-free Coxeter system $(W,S)$ and reduced expression
$\ralpha$ for $w\in W$ such that the induced subgraph $Q_n[X]$ is isomorphic to
$B(\ralpha)$. Are there necessary and sufficient conditions for $X$ to be
admissible?

\item Extend our notion of braid shadow to non-simply-laced Coxeter systems and
prove analogous results to those appearing in Sections~\ref{sec:architecture},
\ref{sec:embedding}, and \ref{sec:fibonacci}. For example, if one generalizes
the notions of braid shadow and link in the natural way, we conjecture that a
result analogous to Proposition~\ref{prop:equalsupport} holds in arbitrary
Coxeter systems as long as the corresponding Coxeter graph does not contain a
three-cycle with edge weights $3, 3, m$, where $m\geq 3$.
\end{enumerate}  

\end{section}

\section*{Acknowledgements} 

We are extremely grateful to Hugh Denoncourt for providing valuable insight on
this project and for assisting with implementing Python code, which we used to
gather data and test conjectures. We are also grateful to the referees for
several useful suggestions.

\bibliographystyle{plain}
\bibliography{BraidClassesSimplyLaced}

\begin{thebibliography}{10}

\bibitem{Bedard1999}
R.~Bedard.
\newblock {On Commutation Classes of Reduced Words in Weyl Groups}.
\newblock {\em European J. Combin}, 20, 1999.

\bibitem{Bergeron2015}
N.~Bergeron, C.~Ceballos, and J.P. Labb{\'{e}}.
\newblock {Fan realizations of subword complexes and multi-associahedra via
  Gale duality}.
\newblock {\em Discrete Comput. Geom.}, 54(1):195--231, 2015.

\bibitem{Cabello2011}
S.~Cabello, D.~Eppstein, and S.~Klav\v{z}ar.
\newblock Structure of {F}ibonacci cubes: a survey.
\newblock {\em Electron. J. Combin.}, 18(1), 2011.

\bibitem{Cadman2021}
Q.~Cadman.
\newblock {\em {Structure of braid graphs in simply-laced Coxeter systems}}.
\newblock {MS Thesis}, Northern Arizona University, 2021.

\bibitem{Cong1993}
B.~Cong, S.~Zheng, and S.~Sharma.
\newblock {On simulations of linear arrays, rings and 2d meshes on Fibonacci
  cube networks}.
\newblock In {\em {Proceedings of the 7th International Parallel Processing
  Symposium}}, pages 747--751. 1993.

\bibitem{Elnitsky1997}
S.~Elnitsky.
\newblock {Rhombic Tilings of Polygons and Classes of Reduced Words in Coxeter
  Groups}.
\newblock {\em J. Combin. Theory, Ser. A}, 77(2), 1997.

\bibitem{Fishel2018}
S.~Fishel, E.~Mili{\'{c}}evi{\'{c}}, R.~Patrias, and B.E. Tenner.
\newblock {Enumerations relating braid and commutation classes}.
\newblock {\em European J. Combin}, 74, 2018.

\bibitem{Geck2000}
M.~Geck and G.~Pfeiffer.
\newblock {\em {Characters of finite Coxeter groups and Iwahori--Hecke
  algebras}}.
\newblock 2000.

\bibitem{Grinberg2017}
D.~Grinberg and A.~Postnikov.
\newblock {Proof of a conjecture of Bergeron, Ceballos and Labb{\'{e}}}.
\newblock {\em New York Journal of Mathematics}, 23:1581--1610, 2017.

\bibitem{Hahn1997}
G.~Hahn and C.~Tardif.
\newblock Graph homomorphisms: structure and symmetry.
\newblock In G.~Hahn and G.~Sabidussi, editors, {\em Graph Symmetry: Algebraic
  Methods and Applications}, pages 107--166. Springer Netherlands, 1997.

\bibitem{Hsu1993}
W.-J. Hsu.
\newblock Fibonacci cubes-a new interconnection topology.
\newblock {\em IEEE Transactions on Parallel and Distributed Systems},
  4(1):3--12, 1993.

\bibitem{Humphreys1990}
J.E. Humphreys.
\newblock {\em {Reflection Groups and Coxeter Groups}}.
\newblock Cambridge University Press, Cambridge, 1990.

\bibitem{klavvzar2005}
S.~Klav\v{z}ar.
\newblock On median nature and enumerative properties of {F}ibonacci-like
  cubes.
\newblock {\em Discrete Math.}, 299(1--3):145--153, 2005.

\bibitem{klavvzar2013structure}
S.~Klav\v{z}ar.
\newblock The {F}ibonacci dimension of a graph.
\newblock {\em Journal of Combinatorial Optimization}, 25(4):505--522, 2013.

\bibitem{Knuth1992}
D.~Knuth.
\newblock {\em {Axioms and Hulls}}.
\newblock Springer-Verlag, Berlin, 1992.

\bibitem{Manin1989}
Yu.I. Manin and V.V. Shekhtman.
\newblock {Arrangements of hyperplanes, higher braid groups and higher Bruhat
  orders}.
\newblock In {\em Explicit universal deformations of Galois representations},
  pages 289--308. Academic Press, Boston, 1989.

\bibitem{Meng2010}
D.~Meng.
\newblock {Reduced decompositions and commutation classes}.
\newblock {\em {\tt arXiv:1009.0886}}, 2010.

\bibitem{ovchinnikov2008partial}
S.~Ovchinnikov.
\newblock Partial cubes: structures, characterizations, and constructions.
\newblock {\em Discrete Mathematics}, 308(23):5597--5621, 2008.

\bibitem{Stanley1984}
R.P. Stanley.
\newblock {On the number of reduced decompositions of elements of Coxeter
  groups}.
\newblock {\em European J. Combin}, 5(4), 1984.

\bibitem{Tenner2006}
B.E. Tenner.
\newblock {Reduced decompositions and permutation patterns}.
\newblock {\em J. Algebraic Combin.}, 24(3):263--284, 2006.

\bibitem{wang2020structure}
Xu~Wang, Xuxu Zhao, and Haiyuan Yao.
\newblock Structure and enumeration results of matchable {L}ucas cubes.
\newblock {\em Discrete Applied Mathematics}, 277:263--279, 2020.

\bibitem{Ziegler1993}
G.M. Ziegler.
\newblock {Higher Bruhat orders and cyclic hyperplane arrangements}.
\newblock {\em Topology}, 32(2), 1993.

\bibitem{Zollinger1994a}
D.M. Zollinger.
\newblock {\em {Equivalence classes of reduced words}}.
\newblock {MS Thesis}, University of Minnesota, 1994.

\end{thebibliography}

\end{document}